\documentclass{amsart}
\usepackage{graphicx} 
\usepackage{amssymb}

\usepackage{amsmath}

\usepackage{subfig}
\usepackage{tikz}
\usetikzlibrary{positioning,quotes,shapes.multipart}

\usepackage{url}

\usepackage{apptools}

\usepackage{url}

\usepackage{apptools}

\usepackage{graphicx} 
\usepackage[initials]{amsrefs}
\usepackage{amsfonts}
\usepackage{amsthm}
\usepackage{enumitem}
\usepackage{xcolor}

\usepackage{bm, graphicx, hyperref, mathrsfs,indentfirst}
\usepackage{stmaryrd}

\usepackage{todonotes}

\usepackage{comment}

\usepackage{todonotes}

\usepackage{comment} 
\newtheorem{theorem}{Theorem}[section]
\newtheorem{lemma}[theorem]{Lemma}
\newtheorem{corollary}[theorem]{Corollary}

\newtheorem{proposition}[theorem]{Proposition}

\theoremstyle{definition}

\theoremstyle{problem}

\theoremstyle{remark}

\numberwithin{equation}{section}

\newcommand{\Aut}{\mbox{\rm Aut}}

\newcommand{\orb}{\mbox{\rm orbit}}
\newcommand{\calK}{\mathcal{K}}

\begin{document}
\title{Orbit equivalence of Cantor minimal systems }
\author{Su Gao}
\address{School of Mathematical Sciences and LPMC, Nankai University, Tianjin 300071, P.R. China}
\email{sgao@nankai.edu.cn}
\thanks{The first author acknowledges the partial support of his research by the National Natural Science Foundation of China (NSFC) grants 12271263 and 12250710128.}
\author{Ruiwen Li}
\address{School of Mathematical Sciences and LPMC, Nankai University, Tianjin 300071, P.R. China}
\email{rwli@mail.nankai.edu.cn}
\thanks{The second author acknowledges the partial support of his research by the National Natural Science Foundation of China (NSFC) grant 124B2001.}
\author{Yiming Sun}
\address{School of Mathematical Sciences and LPMC, Nankai University, Tianjin 300071, P.R. China}
\email{ymsun@mail.nankai.edu.cn}
\begin{abstract} In this paper we study the descriptive complexity of the topological orbit equvalence relation for some Borel classes of Cantor minimal systems. Specifically, we study the Borel class of all Cantor minimal systems with only finitely many ergodic measures, and show that the orbit equivalence for this class is Borel bireducible with the equivalence relation $=^+$. We prove the same for the subclass of regular $\{0,1\}$-Toeplitz subshifts or that of the uniquely ergodic minimal subshifts. We also study the orbit equivalence for the Borel class of minimal subshifts of finite topological rank. Denote by $R_n$ the orbit equivalence for minimal subshifts of topological rank $n\geq 2$. We prove that for any $n\geq 2$, $R_n$ is virtually countable, i.e., Borel reducible to a  countable Borel equivalence relation. Moreover, $R_2$ is virtually amenable. On the other hand, $R_n$ is not smooth when $n\ge 2$, is not virtually hyperfinite when $n\ge 4$,  and is not virtually treeable when $n\ge 5$. For any $n\geq 2$, our contructions yield uniquely ergodic minimal subshifts of topological rank exactly $n$.
\end{abstract}

\maketitle

\section{Introduction}

In this paper we study the complexity of the topological orbit equivalence relation on some Borel classes of Cantor minimal systems from the point of view of descriptive set theory. 

In descriptive set theory, equivalence relations on standard Borel spaces are compared to each other via the notion of Borel reducibility, thus giving a sense of relative complexity of these equivalence relations. More specifically, if $E, F$ are equivalence relations on standard Borel spaces $X, Y$, respectively, then we say that $E$ is {\em Borel reducible} to $F$, denoted $E\leq_B F$, if there is a Borel map $f\colon X\to Y$ such that for any $x_1, x_2\in X$, 
$$ x_1Ex_2\iff f(x_1)Ff(x_2). $$
We say that $E$ is {\em Borel bireducible} with $F$, denoted $E\sim_B F$, if both $E\leq_B F$ and $F\leq_B E$. Intuitively, if $E\leq_B F$, then we regard the complexity of $E$ to be no more than that of $F$, and if $E\sim_B F$, then $E$ and $F$ are thought of having the same complexity. Given an equivalence relation $E$, the objective of the descriptive set theory is to find a well-understood equivalence relation $F$ and establish that $E\sim_B F$; this would completely determine the complexity of $E$. Short of doing this, the next best thing is to find some upper and lower bounds for the complexity of $E$, namely to find some well-understood equivalence relations $F_1$ and $F_2$ and establish $F_1\leq_B E \leq_B F_2$; this would also give much information about the complexity of $E$. For an account of the descriptive set theory of equivalence relations, as well as the definitions of many well-understood equivalence relations, known as {\em benchmark} equivalence relations, the reader is referred to \cite{GaoBook}.

The equivalence relations we study in this paper are the orbit equivalence for various classes of Cantor minimal systems. A {\em Cantor system} is a pair $(X, \varphi)$, where $X$ is a Cantor set and $\varphi\colon X\to X$ is an autohomeomorphism on $X$. For an autohomeomorphism $\varphi\colon X\to X$ and $x\in X$, the {\em orbit} of $x$ under $\varphi$, denoted $\orb_\varphi(x)$, is the set $\{\varphi^k(x)\colon k\in\mathbb{Z}\}$. $\varphi$ is {\em minimal} if for any $x\in X$, $\orb_\varphi(x)$ is dense in $X$. If $\varphi$ is minimal, the Cantor system $(X, \varphi)$ is called a {\em Cantor minimal system}. Given two Cantor minimal systems $(X, \varphi)$ and $(Y, \psi)$, we say that $(X, \varphi)$ and $(Y, \psi)$ are {\em orbit equivalent} if there is a homeomorphism $f\colon X\to Y$ such that for all $x\in X$, 
$$ f(\orb_\varphi(x))=\orb_\psi(f(x)). $$

Fundamental work on the orbit equivalence of all Cantor minimal systems has been done by Giordano, Putnam and Skau \cite{GPS}. In particular, they relate the orbit equivalence to the invariant probability measures of the Cantor minimal systems in the following celebrated theorem.

\begin{theorem} [Giordano--Putnam--Skau {\cite[Theorem 2.2]{GPS}}]\label{thm:GPS} Let $(X, \varphi)$ and $(Y, \psi)$ be Cantor minimal systems. Then the following are equivalent:
\begin{enumerate}
\item[\rm (1)] $(X, \varphi)$ and $(Y, \psi)$ are orbit equivalent.
\item[\rm (2)] The dimension groups $K^0(X, \varphi)/\mbox{Inf}(K^0(X, \varphi))$ and $K^0(Y,\psi)/\mbox{Inf}(K^0(Y,\psi))$ are isomorphic as dimension groups with order units.
\item[\rm (3)] There is a homeomorphism $f\colon X\to Y$ carrying the $\varphi$-invariant probability measures on $X$ onto the $\psi$-invariant probability measures on $Y$.
\end{enumerate}
\end{theorem}

The class of all Cantor minimal systems is a standard Borel space (see Section~\ref{sec:2} for details). Melleray \cite{Melleray} has completely determined the complexity of the orbit equivalence on these classes with the following result.

\begin{theorem} [Melleray {\cite[Theorem 5.9]{Melleray}}]\label{thm:Melleray} The following equivalence relations are Borel bireducible with each other:
\begin{enumerate}
\item[\rm (i)] The orbit equivalence of all Cantor minimal systems.
\item[\rm (ii)] The orbit equivalence of all Toeplitz subshifts.
\item[\rm (iii)] The orbit equivalence of all $\{0,1\}$-Toeplitz subshifts.
\item[\rm (iv)] The universal orbit equivalence relation $E_{S_\infty}$ induced by an action of $S_\infty$, where $S_\infty$ is the group of all permutations of $\mathbb{N}$.
\end{enumerate}
\end{theorem}

 In this paper we consider the class of all Cantor minimal systems with finite topological rank. This is a Borel subclass of all Cantor minimal systems, hence is a standard Borel space itself. We will show that for any postive integer $K$, a Cantor minimal system of topological rank $K$ has at most $K$ many ergodic invariant probability measures. In view of this, we will also consider the Borel class of all Cantor minimal systems with only finitely many ergodic invariant probability measures. We completely determine the complexity of the orbit equivalence on this class with the following result.

\begin{theorem} \label{thm:main1} The following equivalence relations are Borel bireducible with each other:
\begin{enumerate}
\item[\rm (i)] The orbit equivalence of all Cantor  minimal systems with finitely many ergodic invariant probability measures.
\item[\rm (ii)] The orbit equivalence of all uniquely ergodic Cantor minimal systems.
\item[\rm (iii)] The orbit equivalence of all uniquely ergodic minimal subshifts.
\item[\rm (iv)] The orbit equivalence of all regular Toeplitz subshifts.
\item[\rm (v)] The orbit equivalence of all regular $\{0,1\}$-Toeplitz subshifts.
\item[\rm (vi)] The equivalence relation $=^+$ on $(2^{\mathbb{N}})^{\mathbb{N}}$ defined as follows: for any $(x_n), (y_n)\in (2^{\mathbb{N}})^{\mathbb{N}}$, 
$$ (x_n)=^+(y_n)\iff \left\{x_n\in 2^{\mathbb{N}}\colon n\in \mathbb{N}\right\}=\left\{y_n\in 2^{\mathbb{N}}\colon n\in \mathbb{N}\right\}. $$
\end{enumerate}
\end{theorem}
The Borel reducibility from (ii) to (vi) is essentially proved in \cite[Corollary 1]{GPS}.

\begin{figure}[!htbp]

\begin{tikzpicture}[scale=0.045]

\fill[gray!22] (360,2) to (360,196) to (244,196) to (244, 174) to (324,174) to (324, 2);
\fill[gray!22] (252,196) to (252, 139) to (260, 139) to (260, 196) to (252,196);
\fill[gray!22] (282,196) to (282, 99) to (290, 99) to (290, 196) to (282,196);
\node at (345, 123) [text width=1.5cm] {\small Toeplitz subshifts};

\draw (322,0) to (400,0) to (400, 204) to (196,204) to (196,46) to (322,46) to (322,0);
\node at (255,210) {finite topological rank};
\node at (365,214) [text width=3cm] {not of finite topological rank};
\draw[dotted] (320,-5) to (320, 220);
\node at (140,125) [rotate=-90] {finitely many ergodic measures};
\node at (178,180) [text width=1.8cm] {\small uniquely ergodic};
\node at (180, 150) [text width=2cm] {\small 2 ergodic measures};
\node at (180, 110) [text width=2cm] {\small $K$ ergodic measures};
\node at (195, 25) {infinitely many ergodic measures};
\draw[dotted] (138, 44) to (405,44);
\draw[dotted] (148, 162) to (405,162);
\draw[dotted] (148, 137) to (405,137);
\draw[dotted] (148, 122) to (405,122);
\draw[dotted] (148, 97) to (405, 97);

\draw (202,198) to (202,164) to (240,164) to (240, 198) to (202,198);
\node at  (221, 181) {\small odometers};
\draw (202, 160) to (202, 139) to (260, 139) to (260, 198) to (242, 198) to (242, 160) to (202, 160);
\node at (238, 149) [text width=3cm] {\tiny minimal subshifts of topological rank 2};
\draw (202, 120) to (202, 99) to (290, 99) to (290, 198) to (272, 198) to (272, 120) to (202, 120);
\node at (238, 109) [text width=3cm] {\tiny minimal subshifts of topological rank $K$};

\draw (200,200) to (200,50) to (318,50) to (318, 200) to (200, 200);
\draw (198,202) to (198,48) to (398,48) to (398, 202) to (198,202);
\draw[dashed] (244, 196) to (244, 174) to (360, 174) to (360, 196) to (244,196);
\node at (320,185) [text width=3cm] {\tiny regular Toeplitz subshifts};
\draw[dashed] (324, 42) to (324, 2) to (360, 2) to (360,42) to (324, 42);

\end{tikzpicture}
\caption{Classes of Cantor minimal systsems considered in this paper.\label{fig:1}}
\end{figure}
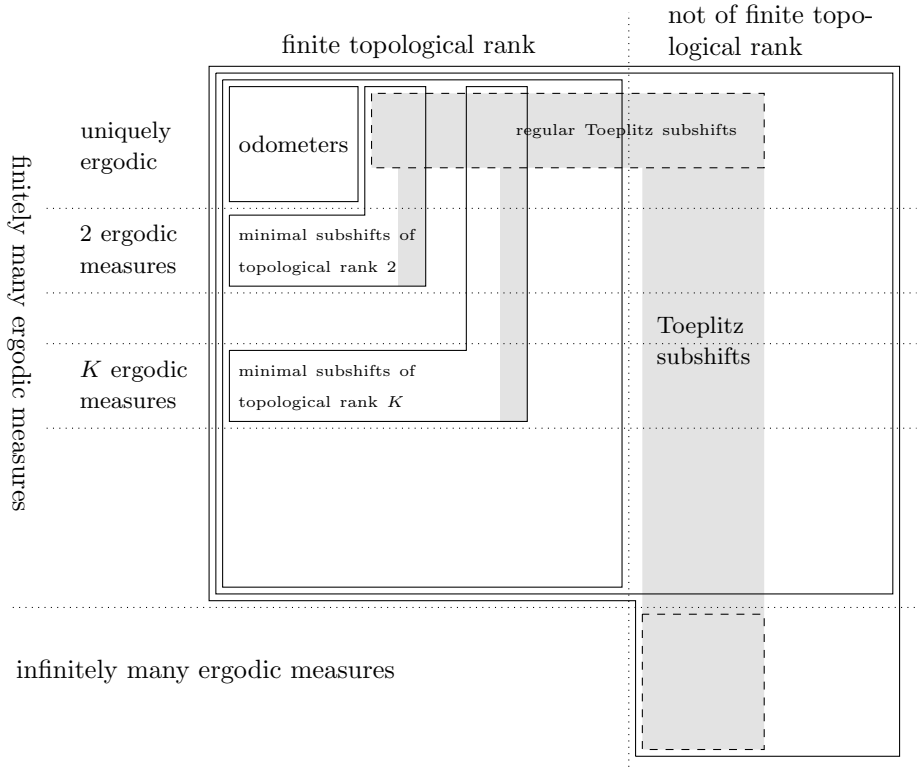

Regarding Cantor minimal systems with finite topological rank, Downarowicz and Maass \cite{DM} have shown that they are topologically conjugate to either odometers (in case of topological rank 1) or minimal subshifts. For the odometers, it is easy to see that the orbit equivalence coincides with the topological conjugacy, and both are Borel bireducible with the equality relation $=$ on $2^{\mathbb{N}}$. Such equivalence relations are called {\em smooth}. 
For minimal subshifts of  finite topological rank, we prove the following result.

\begin{theorem} \label{thm:main2} For any integer $K\ge 2$, the orbit equivalence of minimal subshifts of rank $K$ is virtually countable, i.e., Borel reducible to a countable equivalence relation.
\end{theorem}

We also obtain the following result using a recent theorem of Poulin \cite[Corollary 1.3]{Poulin}.

\begin{theorem}\label{thm:main3}  For any integer $K\ge2$, the orbit equivalence of minimal subshifts of topological rank $K$ is not smooth. Moreover, when $K\geq 4$, it is not virtually hyperfinite, and when $K\geq 5$, it is not virtually treeable.
\end{theorem}

These results give some upper and lower bounds for the orbit equivalence relations in question.  In Figure~\ref{fig:1} we summarize the classes of Cantor minimal systems that are considered in this paper. In Figure~\ref{fig:2} we summarize some of the complexity results.

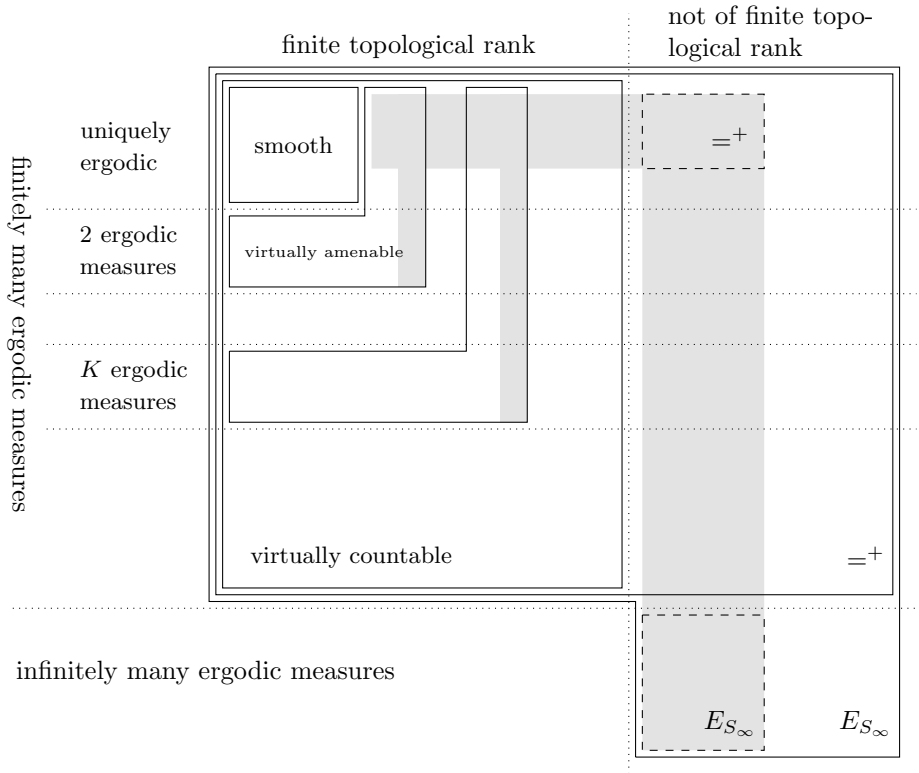
\begin{figure}[!htbp]
\begin{tikzpicture}[scale=0.045]

\fill[gray!22] (360,2) to (360,196) to (244,196) to (244, 174) to (324,174) to (324, 2);
\fill[gray!22] (252,196) to (252, 139) to (260, 139) to (260, 196) to (252,196);
\fill[gray!22] (282,196) to (282, 99) to (290, 99) to (290, 196) to (282,196);

\draw (322,0) to (400,0) to (400, 204) to (196,204) to (196,46) to (322,46) to (322,0);
\node at (255,210) {finite topological rank};
\node at (365,214) [text width=3cm] {not of finite topological rank};
\draw[dotted] (320,-5) to (320, 220);
\node at (140,125) [rotate=-90] {finitely many ergodic measures};
\node at (178,180) [text width=1.8cm] {\small uniquely ergodic};
\node at (180, 150) [text width=2cm] {\small 2 ergodic measures};
\node at (180, 110) [text width=2cm] {\small $K$ ergodic measures};
\node at (195, 25) {infinitely many ergodic measures};
\draw[dotted] (138, 44) to (405,44);
\draw[dotted] (148, 162) to (405,162);
\draw[dotted] (148, 137) to (405,137);
\draw[dotted] (148, 122) to (405,122);
\draw[dotted] (148, 97) to (405, 97);

\draw (202,198) to (202,164) to (240,164) to (240, 198) to (202,198);
\node at (221,181) {\small smooth};
\draw (202, 160) to (202, 139) to (260, 139) to (260, 198) to (242, 198) to (242, 160) to (202, 160);
\node at (230,149) {\tiny virtually amenable};
\draw (202, 120) to (202, 99) to (290, 99) to (290, 198) to (272, 198) to (272, 120) to (202, 120);
\node at (238, 59) {\small virtually countable};

\draw (200,200) to (200,50) to (318,50) to (318, 200) to (200, 200);
\draw (198,202) to (198,48) to (398,48) to (398, 202) to (198,202);
\draw[dashed] (324, 196) to (324, 174) to (360, 174) to (360, 196) to (324,196);

\node at (350, 183) {$=^+$};
\draw[dashed] (324, 42) to (324, 2) to (360, 2) to (360,42) to (324, 42);

\node at (350,10) {$E_{S_\infty}$};
\node at (390,10) {$E_{S_\infty}$};
\node at (390, 59) {$=^+$};

\end{tikzpicture}
\caption{Known results on the complexity of the orbit equivalence.\label{fig:2}}
\end{figure}

In our proofs we come across some constructions of minimal subshifts of topological rank exactly $K$ for any integer $K\geq 2$. This seems to be new and we state it explicitly.

\begin{theorem} For each integer $K\geq 2$ there exist uniquely ergodic minimal subshifts of topological rank exactly $K$.
\end{theorem}

The rest of the paper is organized as follows. In Section~\ref{sec:2} we present some preliminary definitions and theorems. In Section~\ref{sec:3} we prove Theorem~\ref{thm:main1}. In Section~\ref{sec:4} we prove Theorems~\ref{thm:main2} and \ref{thm:main3}.

{\sc Acknowledgments.} We would like to thank Julien Melleray and Bo Peng for useful discussions on the topic of this paper. This research was done when the second author visited McGill University. The second author would like to thank McGill University and Professor Marcin Sabok for their hospitality and support.

\section{Preliminaries \label{sec:2}}

\subsection{Descriptive set theory of equivalence relations}
In this subsection we recall some basic notions and results in the descriptive set theory of equivalence relations. For other undefined notions and more details, the reader can consult \cite{KechrisBook} and \cite{GaoBook}. 

A topological space $X$ is {\em Polish} if $X$ is separable and completely metrizable. A {\em $G_\delta$ subset} of a topological space is the intersection of countably many open sets. In a Polish space $X$, a subspace $Y$ is Polish if and only if $Y$ is a $G_\delta$ subset of $X$. A subset $Y$ of $X$ is {\em comeager} or {\em generic} if $Y$ contains a dense $G_\delta$ subset of $X$; a subset of $X$ is {\em meager} if its complement is comeager. 

If $X$ is a Polish space and $\mathcal{B}$ is the $\sigma$-algebra generated by the open sets of $X$,  then sets in $\mathcal{B}$ are called {\em Borel sets}. More abstractly, a {\em Borel space} is a pair $(X, \mathcal{B})$ where $X$ is a set and $\mathcal{B}$ is a $\sigma$-algebra of subsets of $X$; sets in $\mathcal{B}$ are again called {\em Borel sets}. A {\em standard Borel space} is a Borel space $(X, \mathcal{B})$ where there exists a Polish topology $\tau$ on $X$ such that $\mathcal{B}$ is the $\sigma$-algebra generated by $\tau$. We will use the following fact (see \cite[Corollary 13.4]{KechrisBook}): If $(X, \mathcal{B})$ is a standard Borel space and $Y\in \mathcal{B}$, then $(Y, \mathcal{B}\upharpoonright Y)$ is again a standard Borel space, where 
$$ \mathcal{B}\upharpoonright Y=\{ B\cap Y\colon B\in\mathcal{B}\}. $$
For notational simplicity, we write $X$ for a standard Borel space $(X, \mathcal{B})$ if there is no danger of confusion.

If $X, Y$ are Polish spaces and $f\colon X\to Y$ is a map, then we say that $f$ is {\em continuous} if for any open subset $U$ of $Y$, $f^{-1}(U)$ is an open subset of $X$. If $X, Y$ are Polish spaces or standard Borel spaces and $f\colon X\to Y$ is a map, then we say that $f$ is {\em Borel} (or {\em Borel measurable}) if for any Borel subset $B$ of $Y$, $f^{-1}(B)$ is a Borel subset of $X$. If $X$ is a Polish space or a standard Borel space, a subset $A$ of $X$ is {\em analytic} if there is a standard Borel space $Z$ and a Borel map $g\colon Z\to X$ such that $A=g(Z)$; $A$ is {\em coanalytic} if $X\setminus A$ is analytic. 

An equivalence relation $E$ on $X$ is said to be {\em Borel} (or {\em analytic}) if $E$ is a Borel (or  analytic, respectively) subset of $X\times X$. As we already recalled in the Introduction, if $E, F$ are equivalence relations on Polish or standard Borel spaces $X, Y$, respectively, then $E$ is {\em Borel reducible} to $F$, denoted $E\leq_B F$, if there is a Borel map $f\colon X\to Y$ such that for any $x_1, x_2\in X$, we have
$$ x_1Ex_2\iff f(x_1)Ff(x_2). $$
We say that $E$ is {\em Borel bireducible} with $F$, denoted $E\sim_B F$, if both $E\leq_B F$ and $F\leq_B E$. We say that $E$ is {\em strictly Borel reducible} to $F$, denoted $E<_B F$, if $E\leq_BF$ but $F\not\leq_B E$.

A topological group $G$ is {\em Polish} if its topology is Polish. If $G$ is a Polish group and $X$ is a Polish space, then an action $a\colon G\times X\to X$ of $G$ on $X$ is {\em continuous} if $a$ is a continuous map from the Polish space $G\times X$ to the Polish space $X$. If $G$ is a Polish space and $X$ is a Polish space or a standard Borel space, then an action $a\colon G\times X \to X$ is {\em Borel} if $a$ is a Borel map from $G\times X$ to $X$. For notational simplicity, we use $G\curvearrowright_a X$ or $G\curvearrowright X$ to denote the action $a$, and use $g\cdot x$ to denote $a(g, x)$.

If $G\curvearrowright X$ is a continuous or Borel action of a Polish group $G$ on a Polish or standard Borel space $X$, then for any $x\in X$, the {\em orbit} of $x$ is the set $G\cdot x=\orb_G(x)=\{g\cdot x\colon g\in G\}$. The {\em orbit equivalence relation} induced by the action $G\curvearrowright X$ is defined as
$$ xE^X_Gy\iff G\cdot x=G\cdot y\iff \orb_G(x)=\orb_G(y)\iff \exists g\in G\ (g\cdot x=y). $$
In general, the equivalence relation $E^X_G$ is analytic and not necessarily Borel.

It is a theorem of Becker and Kechris (see, e.g. \cite[Theorem 3.3.4]{GaoBook}) that for any Polish group $G$ there exists a Borel action $G\curvearrowright X$, where $X$ is a standard Borel space, such that for any Borel action $G\curvearrowright Y$, where $Y$ is a standard Borel space, we have $E^Y_G\leq_B E^X_G$. The orbit equivalence relation $E^X_G$ is called a {\em universal} $G$-orbit equivalence relation, and is denoted $E^\infty_G$, or simply $E_G$. If $H$ is a closed subgroup of the Polish group $G$, then $H$ is also Polish and $E_H\leq_B E_G$. 

A Polish group that is the most relevant to this paper is $S_\infty$, the group of all permutations of $\mathbb{N}$. Equipped with the subspace topology of $\mathbb{N}^\mathbb{N}$, $S_\infty$ is a Polish group. The universal $S_\infty$-orbit equivalence relation $E_{S_\infty}$ is one of the most well-studied benchmark equivalence relation in the descriptive set theory of equivalence relations. It is known to be analytic and not Borel, and Borel bireducible with many natural equivalence relations such as the isomorphism relation for all countable graphs, groups, fields, etc. (these are results due to Friedman and Stanley \cite{FS}; also see \cite[Chapter 13]{GaoBook}). In particular, the isomorphism of any Borel class of countable structures is Borel reducible to $E_{S_\infty}$. Thus, one direction of Melleray's Theorem~\ref{thm:Melleray} (the Borel reduction to (v)) follows immediately from Theorem~\ref{thm:GPS} (2). 

The equivalence relation $=^+$ defined in Theorem~\ref{thm:main1} (viii) can easily seen to be an $S_\infty$-orbit equivalence relation. However, it is Borel, hence is strictly Borel reducible to $E_{S_\infty}$. 

If the Polish group $G$ is a discrete countable group, then $E_G$ is always a Borel equivalence relation. Moreover, all of the orbits of $E_G$ is countable. Such equivalence relations are called {\em countable}; namely, a Borel equivalence relation $E$ is {\em countable} if every $E$-equivalence class is countable. Among all countable Borel equivalence relations, $E_{\mathbb{F}_2}$ is the universal one, i.e. $E\leq_B E_{\mathbb{F}_2}$ for any countable Borel equivalence relation, where $\mathbb{F}_2$ is the free group with two generators equipped with the discrete topology. This is a theorem of Jackson, Kechris and Louveau \cite{JKL} (also see \cite[Theorem 7.3.8]{GaoBook}). It follows from results of Hjorth, Kechris and Louveau \cite{HKL} that $E_{\mathbb{F}_2}<_B\ =^+$. Thus $E<_B\ =^+$ for any countable Borel equivalence relation. 

Borel reducibility among countable Borel equivalence relations is very complicated. In this paper we only mention the following basic concepts. First, the simplest, as well as the least complex, Borel equivalence relation with uncountably many equivalence classes is the equality relation on $2^\mathbb{N}$, which we denote by $=$. An equivalence relation is {\em smooth} if it is Borel reducible to $=$. Thus smooth equivalence relations are the least complex in the Borel reducibility hierarchy among Borel equivalence relations with uncountably many equivalence classes. Among nonsmooth Borel equivalence relations, the least complex one is the equivalence relation $E_0$ on $2^\mathbb{N}$ defined as
$$ xE_0 y\iff \exists n\ \forall m>n\ x(m)=y(m). $$
This is a celebrated theorem of Harrington, Kechris and Louveau \cite{HaKL}. $E_0$ is clearly a countable Borel equivalence relation. We call a countable Borel equivalence relation {\em hyperfinite} if it is Borel reducible to $E_0$.

We have the following reducibility relations among the equivalence relations considered in this section:
$$ =\ \ <_B\ E_0\ <_B\ E_{\mathbb{F}_2}\ <_B\ =^+\ <_B\ E_{S_\infty}. $$

Next we review the notion of amenability. A countable group $G$ is {\em amenable} if there is a left-invariant finitely additive probability measure on $G$.  One can also define a notion of amenability for countable Borel equivalence relations as follows. Suppose that $\mu$ is a Borel measure on a standard Borel space $X$ and $E$ is a countable Borel equivalence relation on $X$. We say that $E$
 is {\em $\mu$-amenable} if there is a $\mu$-measurable assignment $x\mapsto \nu_x$ such that:
\begin{itemize}
\item each $\nu_x$ is a finitely additive probability measure on $X$ satisfying that $\nu_x([x]_E)=1$, where $[x]_E$ is the $E$-equivalence class of $x$;
\item if $xEy$ then $\nu_x=\nu_y$. 
\end{itemize}
Here the $\mu$-meaurability of the assignment $x\mapsto \nu_x$ means that for any Borel subset $B$ of $X\times X$, the map
$$ x\mapsto \nu_x(\{ y\in X\colon (x,y)\in B\}) $$
is $\mu$-measurable. Now, a countable Borel equivalence relation $E$ on $X$ is {\em amenable} if it is $\mu$-amenable with respect to every
 Borel measure $\mu$ on $X$. If $G$ is an amenable countable group and $G\curvearrowright X$ is a Borel action, then $E^X_G$ is an amenable, countable Borel equivalence relation. It is known that $E_{\mathbb{F}_2}$ is not amenable (see, e.g., \cite[Proposition 7.4.4 and Theorem 7.4.8]{GaoBook}). Thus, if $E$ is a nonsmooth, amenable countable Borel equivalence relation, then $E_0\leq_B E<_B E_{\mathbb{F}_2}$. 
 
 Finally, a countable Borel equivalence relation $E$ on a standard Borel space $X$ is {\em treeable} if there is a Borel assignment $x\mapsto T_x$ such that 
 \begin{itemize}
 \item each $T_x$ is an acyclic graph (tree) with $[x]_E$ as the vertex set;
 \item if $xEy$ then $T_x=T_y$.
 \end{itemize}
 Here the Borelness of the assignment $x\mapsto T_x$ means that the edge relation $\{(x,y,z)\in X^3\colon \{y,z\}\in T_x\}$ is Borel. Hyperfinite equivalence relations are treeable, and $E_{\mathbb{F}_2}$ is not treeable (\cite[Corollary 3.28]{JKL}). A recent result of Naryshkin and Vaccaro \cite{NV} states that if an orbit equivalence relation $E$ is induced by a Borel action of a countable amenable group, then $E$ is hyperfinite if and only if $E$ is treeable. 

For a Bore equivalence relation $E$ which is not necessarily countable, and for a property $\mathsf{P}\in \{\mbox{countable}, \mbox{treeable}, \mbox{amenable}, \mbox{hyperfinite}\}$, we say that $E$ is {\em virtually $\mathsf{P}$} if $E$ is Borel reducible to a countable Borel equivalence relation with property $\mathsf{P}$. This is a new terminology which was previously called essentially $\mathsf{P}$ in the literature. Following \cite[Section 4.1]{KechrisBookNew}, we now say that $E$ is {\em essentially $\mathsf{P}$} if $E$ is Borel bireducible with a countable Borel equivalence relation with property $\mathsf{P}$. It is known that essentially countable and virtually countable are different concepts (\cite{Hj}). However, by the Harrington--Kechris--Louveau dichotomy theorem (see, e.g., \cite[Theorem 6.3.1]{GaoBook}), a Borel equivalence relation is virtually hyperfinite if and only if it is essentially hyperfinite.

\subsection{Standard Borel spaces of Cantor minimal systems}
In this subsection we establish some notation and basic results about the space of all Cantor minimal systems. Our treatment follows \cite{GL} unless otherwise indicated.

A {\em Cantor set} is a $0$-dimensional compact Polish space without isolated points. It is well-known that all Cantor sets are homeomorphic to $\mathcal{C}=2^{\mathbb{Z}}=\{0,1\}^\mathbb{Z}$. Let $\Aut(\mathcal{C})$ be the group of all autohomeomorphisms of $\mathcal{C}$ where the group operation is composition. Equipped with the supnorm metric, $\Aut(\mathcal{C})$ becomes a Polish group. Let $M(\mathcal{C})$ be the set of all minimal homeomorphisms in $\Aut(\mathcal{C})$. Then $M(\mathcal{C})$ is a $G_\delta$ subset of $\Aut(\mathcal{C})$, and hence is a Polish space itself. We view $M(\mathcal{C})$ as the space of all Cantor minimal systems. The orbit equivalence of Cantor minimal systems is thus an analytic equivalence relation on $M(\mathcal{C})$. It follows from Melleray's Theorem~\ref{thm:Melleray} that this equivalence relation is not Borel.

The set of all odometers is a dense $G_\delta$ subset of $M(\mathcal{C})$ (see, e.g., \cite[Proposition 3.6]{GL} or \cite[Lemma 4.1]{Melleray}). More generally, for each positive integer $K$, the set of all Cantor minimal systems of topological rank at most $K$ is also a $G_\delta$ subset of $M(\mathcal{C})$ (see, e.g., \cite[Corollary 3.4]{GL}). 
Thus, for any positive integer $K$, the set of all Cantor minimal systems of topological rank $K$ is a Borel subset of $M(\mathcal{C})$, and hence is a standard Borel space. It follows that the set of all Cantor minimal systems of finite topological rank is a standard Borel space as well.

By a theorem of Downarowicz and Maass \cite{DM}, there are only two kinds of Cantor minimal systems of finite topological rank; when it has topological rank $1$, it is an odometer, and when it has topological rank $K>1$, it is topologically conjugate to a minimal subshift. We recall the definition of subshifts. Let $\mathsf{A}$ be a finite set (which we call the {\em alphabet}). The space $\mathsf{A}^\mathbb{Z}$ is a Cantor set. Consider the {\em shift action} $\mathbb{Z}\curvearrowright \mathsf{A}^\mathbb{Z}$: for any $g\in \mathbb{Z}$ and $x\in \mathsf{A}^\mathbb{Z}$,
$$ \varphi_g(x)(n)=(g\cdot x)(n)=x(n+g) $$
for any $n\in\mathbb{Z}$. Then $(\mathsf{A}^\mathbb{Z}, \varphi_1)$ is a Cantor system. Given a Cantor system $(X, \varphi)$, a subset $A$ of $X$ is {\em $\varphi$-invariant} if $\varphi(A)=A$. A {\em subshift} $X$ is a closed $\varphi_1$-invariant subset of $\mathsf{A}^\mathbb{Z}$. If $X$ is a subshift, then $(X, \varphi_1\upharpoonright X)$ is a Cantor system. 

Since the set of all odometers is a comeager subset of $M(\mathcal{C})$, the set of all minimal subshifts is thus a meager subset of $M(\mathcal{C})$. Following Pavlov and Schmieding \cite{PS}, one can consider a suitable Polish space of subshifts which include all minimal subshifts (we do not give the details here since we do not use them in this paper). In this Polish space, the set of all minimal subshifts is a comeager Borel subset (see \cite[Theorem 1.3]{PS}), hence it is possible to both consider it as a standard Borel space and speak of genericity among minimal subshifts. According to \cite[Theorem 1.3]{PS}, the set of all regular Toeplitz subshifts of topological rank $2$ is a comeager set of minimal subshifts. Each of the concepts in this statement corresponds to a Borel set. Moreover, it is shown in \cite[Theorem 1.3]{PS} that a generic minimal subshift is orbit equivalent to the universal odometer. This is in contrast with our results (Theorems~\ref{thm:main2} and \ref{thm:main3}) that for any $K>1$, the orbit equivalence of minimal subshifts of topological rank $K$ is not smooth.

Given a dynamical systems $(X, \varphi)$,  a Borel measure $\mu$ on $X$  is {\em $\varphi$-invariant} if $\mu(B) =\mu(\varphi^{-1}(B))$ for all Borel subsets $B$ of $X$. $\mu$ is {\em ergodic} if for every $\varphi$-invariant Borel set $B $, either $\mu(B)=0$ or $\mu(X\setminus B)=0$.  The  ergodicity of $\mu$ is equivalent to the following: if $\mu=a\mu_1+(1-a)\mu_2$  for two measures and some $0<a<1$, then $\mu=\mu_1=\mu_2$.

Given a Cantor space $X$, we denote by $\mbox{Prob}(X)$ the set of Borel probability measures on  $X$.   Note  that a Borel probability measure on $X$ is uniquely determined by its values on clopen sets, and that any finitely additive probability measure on $\mbox{Clopen}(X)$ extends to an element of  $\mbox{Prob}(X)$. Let $C(X)$ denotes the space of all real-valued continuous functions on $X$. Let $C(X)^*$ be the dual space of $C(X)$ equipped  with the weak*  topology. We can regard  $\mbox{Prob}(X)$ as a subspace of $C(X)^*$, where the inherited topology on $\mbox{Prob}(X)$ is a natural compact topology induced by the maps $\mu \mapsto \mu(U)$,  where $U$ ranges over all clopen subsets of $X$.  Define a metric on $\mbox{Prob}(X)$ as 
$$d(\mu, \nu)=\sup_{f\in C(X)}\left| \int_X  f\,d\mu- \int_X  f\, d\nu\right|$$ for any $\mu, \nu \in \mbox{Prob}(X)$.  This is a  compatible metric for the weak* topology, which makes $\mbox{Prob}(X)$ a compact Polish space.

Given a Cantor system $(X,\varphi)$, the set of all $\varphi$-invariant Borel probability measures on $X$ is nonempty, and we denote it by $\mbox{M}_\varphi$. Note that $\mbox{M}_\varphi$ is compact subset of $\mbox{Prob}(X)$. Let $\mbox{ME}_\varphi$ be the set of all ergodic measures on $(X,\varphi)$. Suppose that $(X,\varphi)$ is  a Cantor minimal system with finitely many ergodic measures. Then $\mbox{EM}_\varphi$ is also a compact subset of $\mbox{Prob}(X)$.  Let $\mu_1,\dots, \mu_n$ be  an enumeration of all of its ergodic measures. Then 
for every $\varphi$-invariant measure $\mu\in \mbox{Prob}(X)$,  there are positive  numbers $a_1,\dots, a_n$ such that $a_1+\cdots +a_n=1$ and 
$$\mu=\sum_{i=1}^na_i\mu_i.$$

Given a compact Polish space $Z$, we denote by $\calK(Z)$ the space  of  all compact   subsets of $Z$ equipped with the Vietoris topology.  Then $\calK(Z)$ is again a compact Polish space (see, e.g., \cite[Theorem 4.26]{KechrisBook}).

In the next section we will consider the orbit equivalence of Cantor minimal systems with finitely many ergodic measures. In the following we show that it is a Borel class and that the map that assigns the set of ergodic measures to each such system is a Borel map. 


\begin{proposition}
For any positive integer $K\ge1$, the class of Cantor systems with exactly $K$ ergodic measures is a Borel subset of $\Aut(\mathcal{C})$, and the function from this class to $\calK({\rm Prob}(\mathcal{C}))$ that sends a Cantor system to the set of its ergodic measures is Borel.
\end{proposition}

\begin{proof} Fix a positive integer $K\geq 1$. 
Without loss of generality we may assume $X=\mathcal{C}$.  By \cite[Lemma 4.2]{Melleray}, the map from $\rm{Aut(\mathcal{C})}$ to $\calK(\rm{Prob}(\mathcal{C}))$ that sends $\varphi$ to $ \mbox{\rm M}_\varphi$ is a Borel map.   By the Kuratowski--Ryll-Nardzewski Selection Theorem (see, e.g., \cite[Theorem 12.13]{KechrisBook}), there is a sequence of Borel functions 
$$S_m\colon \calK({\rm Prob}(\mathcal{C}))\to {\rm Prob}(\mathcal{C}) $$
for $m\in\mathbb{N}$ such that for any $F\in \calK(\mbox{\rm Prob}(\mathcal{C})$, $\{S_m(F)\colon m\in\mathbb{N}\}$ is dense in $F$. 

We claim that  $|\mbox{ME}_\varphi|\leq K $  if and only if 
\begin{quote}
there exist $n\leq K$ and $x_1,\cdots, x_n\in \mbox{\rm M}_\varphi$ such that for every positive rational $\epsilon>0$ and every $m\in\mathbb{N}$, there are nonnegative rationals $a_1,\dots, a_n$ such that $a_1+\cdots +a_n=1$ and 
$$d\left(\sum_{i=1}^na_ix_i, S_m(\mbox{\rm M}_\varphi)\right)<\epsilon.$$   
\end{quote}
For the direction of the necessity we only need to take $x_1,\cdots, x_n$ as the ergodic measures of $(X, \varphi)$.  Thus it suffices to show the direction of the sufficiency. Toward a contradiction, assume the displayed condition holds but there are pairwise distinct $y_1,\dots , y_{n+1} \in \rm{ME}_\varphi$. By our hypothesis, each $y_k$, $k=1,\dots, n+1$, can be expressed as a  linear  combination of $x_1, \dots, x_n$ and the coefficients are nonnegative reals that sum to 1. 
By ergodicity, for each $1\leq k\leq n+1$, there is $1\leq i\leq n$ such that $y_k=x_i$. This implies that for some $1\leq j, k\leq n+1$, $j\neq k$, we have $y_j=y_k$, a contradiction. This finishes the proof of the claim. By the claim, the class of Cantor systems with at most $K$ ergodic measures is an analytic subset of $\Aut(\mathcal{C})$.

On the other hand, we claim that $|\mbox{ME}_\varphi|\leq K $  if and only if 
\begin{quote} for any $x_1,\cdots, x_K, x_{K+1}\in \rm M_\varphi$, there exists $1\leq n \leq K+1$ such that for every positive rational $\epsilon>0$, there are nonnegative  rationals $a_1,\cdots, a_{K+1}$ such that $a_n=0$,  $a_1+\cdots+a_{K+1}=1$ and $$d\left(\sum_{i=1}^{K+1}a_ix_i,x_n\right)< \epsilon. $$ 
\end{quote}
The proof of this claim  is similiar as the above.  By this claim  we know that the class of Cantor systems with at most $K$ ergodic measures is a coanalytic subset of $\Aut(\mathcal{C})$. 

Combining the two claims and applying Suslin's Theorem (see, e.g., \cite[Theorem 14.11]{KechrisBook}), we get that the class of Cantor systems with at most $K$ ergodic measures is a Borel subset of $\Aut(\mathcal{C})$. Let $M_K$ denote the subset of $\Aut(\mathcal{C})$ consisting of all Cantor systems with exactly $K$ ergodic measures. Then $M_K$ is also Borel.

Let $\Theta\colon M_K\to \calK({\rm Prob}(\mathcal{C}))$ be the map that sends a Cantor system $(X, \varphi)$ in $M_K$ to the set of its ergodic measures $\mbox{\rm ME}_\varphi$. Then $\Theta$ is Borel if and only if its graph $\{(\varphi, \mbox{\rm ME}_\varphi)\colon \varphi\in M_K\}$ is Borel. Now note that for $\varphi\in M_K$ and $F\in \calK(\mbox{\rm Prob}(\mathcal{C})$, $F=\Theta(\varphi)=\mbox{ME}_\varphi$ if and only if $|F|=|\{S_m(F)\colon m\in\mathbb{N}\}|=K$, $F\subseteq \mbox{\rm M}_\varphi$, and if we enumerate the elements of $F$ as $x_1,\cdots, x_K\in \rm{M}_\varphi$, then for every positive rational $\epsilon>0$ and $m\in\mathbb{N}$, there are positive rationals $a_1,\dots, a_K$ such that $a_1+\cdots +a_K=1$ and 
$$d\left(\sum_{i=1}^Ka_ix_i, S_m(\mbox{\rm M}_\varphi)\right)<\epsilon.$$ 
This characterization is Borel, since for instance $F\subseteq \mbox{\rm M}_\varphi$ is Borel by \cite[Exercise 4.29 (ii)]{KechrisBook}. 
\end{proof}

\begin{corollary} \label{cor:KEM}
For any positive integer $K\ge1$, the class of Cantor minimal systems with exactly $K$ ergodic measures is a Borel subset of $M(\mathcal{C})$, and the function from this class to $\calK({\rm Prob}(\mathcal{C}))$ that sends a Cantor minimal system to the set of its ergodic measures is Borel.
\end{corollary}

\section{Cantor minimal systems with finitely many ergodic measures\label{sec:3}}

In this section we prove Theorem~\ref{thm:main1}.  The classes of Cantor minimal systems considered in Theorem~\ref{thm:main1} are in decreasing order in the sense of inclusion from (i) to (v). Thus it suffices to prove that the orbit equivalence for the dlass in (i) is Borel reducible to $=^+$ and that $=^+$ is Borel reducible to the orbit equivalence for the class in (v).

We will recall the notions and results about Cantor minimal systems as needed.

\subsection{The upper bound}
We first recall some definitions and theorems from \cite{GPS}. 


A countable abelian group $G$ together with a subset $G^+$ is called an {\em ordered group} if  $G^++G^+\subseteq G^+$, $G^+-G^+=G$ and $G^+\cap (-G^+)=\varnothing$. Here $G^+$ is called the {\em positive cone}. We shall write $a \le b$ if $b - a \in G^+$. If $(G, G^+)$ is an ordered group and $H$ is a subgroup of $G$, then we set $H^+=H\cap G^+$. An element $u \in G^+$ is  an  {\em order unit} for $(G,G^+)$ if  for every $a \in G$,  $a \le  nu$ for some $n \in \mathbb N$.  We say that the ordered group $(G,G^+)$ is {\em unperforated} if whenever $a\in G$ and $na\in G^+$ for some $n \in \mathbb N$, we have $a\in G^+$.  

An unperforated ordered group $(G,G^+)$ is a {\em dimension group} if given $a_1$,  $a_2$,  $b_1$,  $b_2 \in G$ with $a_i \le b_j$ for all $i, j =1, 2$, there exists a $c\in G$ with $a_i\le c \le b_j$ for all $i, j =1, 2$.   Let $(G_1, G_1^+, u_1)$ and $(G_2, G_2^+, u_2)$  be  two dimension groups  with distinguished order units.  We say that they are {\em isomorphic}, and write $(G_1, G_1^+, u_1)\cong (G_2, G_2^+, u_2)$, if there exists a group isomorphism $f\colon G_1 \to G_2$  such that  $f(G_1^+)=G_2^+$ and $f(u_1)=u_2$. 

If $(G, G^+)$ is an ordered group, an {\em order ideal} is a subgroup $H$ of $G$ such that $H^+-H^+=H$ and whenever $0\leq a\leq b$ and $b\in H$, we have $a\in H$. A dimension group $(G,G^+)$ is {\em simple}   if $G$  does not contain any nontrivial order ideal. If $G$ is a simple dimension group
it is easily seen that any $a \in G^+\setminus \{0\}$ is an order unit.  Let $G$ be a simple dimension group and let $u \in G^+\setminus \{0\}$.  We say that $a \in G$ is {\em infinitesimal} if for all positive integers $p, q$, we have $-pu\leq qa\leq pu$. Informally, we also write this condition as $ -\epsilon u \le a \le \epsilon u$ for all $0<\epsilon \in \mathbb Q$.  The collection of infinitesimal elements of  $G$ form a subgroup,  {\em the infinitesimal subgroup} of $G$, which we denote by  $\mbox{Inf} (G)$. 

The quotient group $G /\mbox{Inf} (G)$ has a natural induced ordering. Let $a\mapsto \hat{a}$ denote the quotient map from $G$ onto $G/\mbox{Inf}(G)$. Then we can set $\hat{a} > 0$ iff $a > 0$. It is easily seen that $G /\mbox{Inf} (G)$ is a simple dimension group with no infinitesimal elements except 0. If $G$ has distinguished order unit $u$ then $G /\mbox{Inf} (G)$  inherits the distinguished order unit $\hat{u}$.

Let $(X, \varphi)$  be a Cantor minimal system.  Denote by $C(X, \mathbb Z)$ the  set of  the continuous maps from $X$ to $\mathbb Z$. So $C(X, \mathbb Z)$ is a countable abelian group under addition.  Let  $B_ \varphi$  denote the {\em coboundary subgroup} $\{f-f\circ  \varphi^{-1}\colon f \in C(X, \mathbb Z)\}$. Define  $K^0(X, \varphi)$ as the quotient group $C(X, \mathbb Z)/B_ \varphi$ and  let
$$K^0(X, \varphi)^+=\left\{ [f]\colon f  \ge 0, f \in C(X, \mathbb Z)\right\},$$
where $f\mapsto [f]$ denotes the  quotient map. Let $\mathbf 1$ denote the element of $K^0(X, \varphi)^+$  that the constant function $1$ maps to. Then
$K^0(X, \varphi)$ with the positive cone $K^0(X, \varphi)^+$ is a simple,  torsion-free dimension group with  distinguished order unit $\hat{\mathbf 1}$.  




We will use the following result, which was explicitly stated as a part of \cite[Theorem 1.13]{GPS} but can be traced back to \cite[Theorem 5.5]{HPS} and \cite[Corollary 4.2]{EffrosBook}.

\begin{theorem}\label{thm:GPS2} Let $(X,\varphi)$ be a Cantor minimal system. Then
$$\mbox{\rm Inf}\left(K^0(X, \varphi)\right)=\left\{[f]\colon  f\in C(X,\mathbb Z), \int f\,d \mu =0\mbox{ for all } \varphi \mbox{-invariant measures }\mu\right\}.$$
\end{theorem}

Suppose that $(X,\varphi)$ is  a Cantor minimal system with finitely many ergodic measures. Let $\mu_1,\dots, \mu_n$ be a nonrepetitive enumeration of its ergodic measures. Define $\Gamma_\varphi$ as 
$$\Gamma_\varphi= \left\{\left(\int f\,d\mu_1,\dots , \int f\,d\mu_n\right)\colon  f \in C(X,\mathbb Z)\right\}.$$
It is clear that $\Gamma_\varphi$ is a subgroup of $\mathbb R^n$ and 
$$\bigl(\Gamma_\varphi, \Gamma_\varphi\cap\bigl( \mathbb R^+\bigr)^n, 1^n \bigr)$$ is a dimension group with distinguished order unit, where $1^n=(1,\dots, 1)\in \mathbb{R}^n$.

We first note the following direct corollary of Theorem~\ref{thm:GPS}.

\begin{lemma}\label{lem:dim} Let $(X, \varphi)$ be a Cantor minimal system with finitely many ergodic measures $\mu_1,\dots, \mu_n$. Then the following dimension groups are isomorphic:
$$ \Bigl(K^0(X, \varphi)/\mbox{\rm Inf}\left(K^0(X,\varphi)\right), K^0(X, \varphi)^+/\mbox{\rm Inf}\left(K^0(X, \varphi)\right), \hat{\mathbf 1}\Bigr)\cong \bigl(\Gamma_\varphi, \Gamma_\varphi\cap (\mathbb{R}^+)^n, 1^n\bigr). $$
\end{lemma}

\begin{proof}
 Using Theorem~\ref{thm:GPS2}, one readily checks that the map
$$ [f]+\mbox{\rm Inf}\left(K^0(X, \varphi)\right)\mapsto \left(\int f\,d\mu_1,\dots , \int f\,d\mu_n\right) $$
is well defined and injective. 
By \cite[Theorem 1.13]{GPS},
\begin{align*}
K^0(X,\varphi)^+ &
=\left\{[f]\in K^0(X,\varphi)\colon \int f\,d\mu>0 \mbox{ for all }  \mu \in M_{\varphi}\right\}\cup\mbox{Inf}\left(K^0(X,\varphi)\right)\\
& =\left\{[f]\in K^0(X,\varphi)\colon \int f\,d\mu_i> 0 \mbox{ for all }  1\leq i\leq n \right\}\cup\mbox{Inf}\left(K^0(X,\varphi)\right).
\end{align*}
It is then easily seen that the map is a group isomorphism that sends the positive cone onto the positive cone and $\hat{\mathbf 1}$ to $1^n$. 
\end{proof}

If $G$ is  a subgroup of $\mathbb R^n$ and $\phi\in \mbox{\rm Sym}(n)$, we define
$$ \phi(G)=\left\{(r_{\phi^{-1}(1)}, \dots, r_{\phi^{-1}(n)})\colon (r_1, \dots, r_n)\in G\right\}. $$
Then $\phi(G)$ is still a subgroup of $\mathbb{R}^n$. The following is a characterization of orbit equivalence for Cantor minimal systems with finitely many ergodic measures.

\begin{lemma}\label{lem:OEkey}
Let  $(X,\varphi)$ and $(Y,\psi)$ be two Cantor minimal systems with finitely many ergodic measures. Then  $(X,\varphi)$ and $(Y,\psi)$ are orbit equivalent if and only if $(X,\varphi)$ and $(Y,\psi)$ have the same number of  ergodic measures, say $n$,  and there exists $\phi \in \mbox{\rm Sym}(n)$ such that $\Gamma_{\varphi}=\phi(\Gamma_{\psi})$.
\end{lemma}
\begin{proof} Suppose $(X, \varphi)$ and $(Y, \psi)$ are orbit equivalent. By Theorem~\ref{thm:GPS} (3), there is a homeomorphism $f\colon X\to Y$ carrying the $\varphi$-invariant measures on $X$ onto the $\psi$-invariant measures on $Y$. In particular, $f$ carries the ergodic measures for $\varphi$ onto the ergodic measures for $\psi$. Let $n$ be the number of ergodic measures for $\varphi$. Then $f$ induces a $\phi\in \mbox{\rm Sym}(n)$ so that $\Gamma_\varphi=\phi(\Gamma_\psi)$.

Conversely, suppose there is $\phi\in\mbox{\rm Sym}(n)$ such that $\Gamma_\varphi=\phi(\Gamma_\psi)$. Then in particular the groups $\Gamma_\varphi$ and $\Gamma_\psi$ are isomorphic as dimension groups. By Lemma~\ref{lem:dim} and Theorem~\ref{thm:GPS} (2), we have that $(X, \varphi)$ and $(Y, \psi)$ are orbit equivalent.
\end{proof}

The above lemma demonstrates that the orbit equivalence for Cantor minimal systems is a Borel equivalence relation. In the following, we determine its complexity in the Borel reducibility hierarchy.

\begin{theorem}  The orbit equivalence for Cantor minimal systems with finitely many ergodic measures is Borel reducible to $=^+$.
\end{theorem}

\begin{proof} Fix a positive integer $n$ and let $M_n$ be the class of Cantor minimal systems with exactly $n$ many ergodic measures. By Corollary~\ref{cor:KEM}, $M_n$ is a Borel subset of $M(\mathcal{C})$, and hence a standard Borel space. It suffices to show that the orbit equivalence on $M_n$ is Borel reducible to $=^+$. 

Enumerate all continuous functions from $\mathcal{C}$ to $\mathbb{Z}$ as $(f_i)_{i\in\mathbb{N}}$. For $\varphi\in M_n$, enumerate its ergodic measures as $\mu^\varphi_1,\cdots,\mu^\varphi_n$, which can be accomplished in a Borel way by the Kuratowski--Ryll-Nardzewski Selection Theorem. Define $\Phi\colon M_n\to (\mathbb{R}^n)^\mathbb{N}$ by letting
$$ \Phi(\varphi)=\left(\int f_i\,d\mu^\varphi_1,\cdots,\int f_i\,d\mu^\varphi_n\right)_{i\in\mathbb{N}}. $$
Then $\Phi$ is a Borel function.

Define an equivalence relation $E$ on $(\mathbb{R}^n)^\mathbb{N}$ by
$$ (x_{i,1},\cdots,x_{i,n})_{i\in\mathbb{N}} \mbox{ and } (y_{i,1},\cdots,y_{i,n})_{i\in\mathbb{N}} \mbox{ are $E$-equivalent} $$
if and only if there is $\phi\in \mbox{\rm Sym}(n)$ such that 
$$\left\{ (x_{i,1},\cdots,x_{i,n})\colon i\in\mathbb{N}\right\}=\{(y_{i,\phi(1)},\cdots,y_{i,\phi(n)})\colon i\in\mathbb{N}\}. $$

By Lemma~\ref{lem:OEkey}, $\Phi$ witnesses that the orbit equivalence on $M_n$ is Borel reducible to $E$.  Now note that $E$ is ${\bf\Pi}^0_3$ and it is induced by a continuous action of the group $S_\infty\times S(n)$ on $(\mathbb{R}^n)^\mathbb{N}$. By a theorem of Hjorth, Kechris and Louveau \cite[Theorem 2]{HKL} (also see \cite[Theorem 12.5.5]{GaoBook}), $E$ is Borel reducible to $=^+$.
\end{proof}

This theorem gives the upper bound for the complexity of the orbit equivalence for Cantor minimal systems with finitely many ergodic measures.

\subsection{The lower bound} We consider a special kind of Cantor minimal systems, namely Toeplitz subshifts. 

We first recall some basic definitions and results about Toeplitz subshifts following \cite{Williams}. 

Let $\mathsf{A}$ be a finite set. A subshift $X\subseteq \mathsf{A}^\mathbb{Z}$ is {\em Toeplitz} if $X$ is the closure of $\orb_{\varphi_1}(x)$ for some $x\in \mathsf{A}^\mathbb{Z}$, where $x$ satisfies that for any $k\in\mathbb{Z}$ there is $p\in \mathbb{N}$ such that $x(k)=x(k+np)$ for all $n\in\mathbb{Z}$. Such an $x$ is called a {\em Toeplitz sequence}; and we say that $X$ is {\em generated} by $x$. In the context of subshifts we usually denote the shift map $\varphi_1$ as $\sigma$. When we want to emphasize the alphabet, we also call $x$ an {\em $\mathsf{A}$-Toeplitz sequence} and $X$ an {\em $\mathsf{A}$-Toeplitz subshift}. 

Let $x$ be an $\mathsf{A}$-Toeplitz sequence. For $p\in \mathbb{N}$, let
$$ \mbox{\rm Per}_p(x)=\left\{k\in \mathbb{Z}\colon x(k)=x(k+np) \mbox{ for all $n\in \mathbb{Z}$}\right\}. $$
Fix a new symbol $\Box\not\in \mathsf{A}$. Define the {\em $p$-skeleton} of $x$ as an element $x'\in (A\cup\{\Box\})^\mathbb{Z}$ where 
$$ x'(k)=\left\{\begin{array}{ll} x(k), & \mbox{ if $k\in \mbox{\rm Per}_p(x)$, }\\ \Box, & \mbox{ otherwise.}\end{array}\right. $$
Then the $p$-skeleton of $x$ is periodic with period $p$. We call $p$ an {\em essential period} of $x$ if the $p$-skeleton of $x$ is not periodic with any smaller period. Let $P(x)$ denote the set of all essential periods of $x$. For $p, q\in P(x)$, let $p\sqsubseteq q$ if $p$ is a factor of $q$. Then $(P(x), \sqsubseteq)$ is a directed set (see \cite[Proposition 2.1]{Williams}). For each $p\in P(x)$, define
$$ \delta_p(x)=\displaystyle\frac{1}{p}|\{k\in \mathbb{Z}\colon 0\leq k<p, k\in \mbox{\rm Per}_p(x)\}|. $$
Then it is easily seen that $p\sqsubseteq q$ implies $\delta_p(x)\leq \delta_q(x)$. It follows that the limit of $\delta_p(x)$ for $p\in P(x)$ exists (as a limit of a net), and we denote it as $\delta(x)$. $x$ is called {\em regular} if $\delta(x)=1$. We also call the Toeplitz subshift generated by $x$ {\em regular} if $x$ is regular. It is a theorem of Jacobs and Keane (see \cite[Theorem 2.6]{Williams}) that a regular Toeplitz subshift is uniquely ergodic.

It is routine to see that the class of all regular $\{0,1\}$-Toeplitz subshifts is a Borel subset of $M(\mathcal{C})$. It is, of course, also a Borel subset of $M_1$, the class of all uniquely ergodic Cantor minimal systems.

To construct Toeplitz subshifts we will use their $\mathcal{S}$-adic representations, which are studied by Arbul\'{u}, Durand and Espinoza \cite{ADE}. In the following we recall the basic notions of this theory with an adaptation that suits our purpose here. Our presentation is therefore also partially based on \cite{GLPS}. 

Elements of the alphabet $\mathsf{A}$ are called {\em letters}. Let $\mathsf{A}^*$ denote the set of all finite $\mathsf{A}$-sequences (or {\em words}). For each word $u\in \mathsf{A}^*$, let $|u|$ denote the length of $u$. The unique word of length $0$ is the {\em empty word}, which we denote by $\varnothing$. Any word $u$ can be written as a concatenation of letters
$$ u=u(0)u(1)\dots u(|u|-1) $$
where $u(0), u(1), \dots, u({|u|-1})\in\mathsf{A}$. 
If $S\subseteq \mathsf{A}^*$ and $u\in \mathsf{A}^*$, we say that $u$ is {\em built from} $S$ if we can write $u$ as a concatenation of words
$$ u=v_1\cdots v_k $$
where $v_1,\dots, v_k\in S$. The sequence $(v_1, \dots, v_k)$ is called a {\em building} of $u$ from $S$. Note that the ``built from'' relation is transitive, in the sense that if $u$ is built from $S$ and each word in $S$ is built from $T$, then $u$ is built from $T$. If $u$ is built from $S$ and the words in $S$ all have the same length, then there is a unique building of $u$ from $S$, and the occurrences of the terms of the building in $u$ are called {\em expected occurrences}. 

A {\em generating sequence} is a doubly indexed sequence $${\mathbf u}=(u_{n, i})_{n,i\in\mathbb{N}, i\leq \ell_n}$$ for some $(\ell_n)_{n\in\mathbb{N}}\in \mathbb{N}^\mathbb{N}$ satisfying that
\begin{itemize}
\item for any $n\in\mathbb{N}$ and $i\leq \ell_n$, $u_{n,i}\in \mathsf{A}^*$;
\item for any $n\in \mathbb{N}$ and $i,j\leq \ell_n$, if $i\neq j$ then $u_{n,i}\neq u_{n,j}$;
\item for any $n\in \mathbb{N}^+$ and $i\leq \ell_n$, $u_{n,i}$ is built from the set of words
$$S_{n-1}=\{ u_{n-1, j}\colon 0\leq j\leq \ell_{n-1}\}.$$
\end{itemize}
A generating sequence ${\bf u}$ as above is said to have {\em constant length} if for all $n\in\mathbb{N}$ and $i, j\leq \ell_n$, $|u_{n,i}|=|u_{n,j}|$. A generating sequence ${\bf u}$ as above is {\em proper} if for all $n\in\mathbb{N}^+$ there exist $v_0, v_1\in S_{n-1}$ such that for any $i\leq\ell_n$, there is some building of $u_{n,i}$ from $S_{n-1}$ of which $v_0$ is the first term and $v_1$ is the last term. A generating sequence ${\bf u}$ as above is {\em primitive} if for any $n\in\mathbb{N}$ there is $N>n$ such that for all $i\leq \ell_N$, there is a building of $u_{N,i}$ from $S_n$ in which every element of $S_n$ occurs as a term. 

Given a generating sequence ${\bf u}=(u_{n,i})_{n\in\mathbb{N}, i\leq \ell_n}$, we define 
$$\begin{array}{rl} X_{\bf u}=\{x\in \mathsf{A}^\mathbb{Z}\colon &\!\!\!\!\!\!\!\! \mbox{ every finite subword of $x$ is a subword of} \\
& \mbox{ some $u_{n,i}$ for $n\in\mathbb{N}$ and $i\leq \ell_n$}\}. 
\end{array}
$$ 
Then $X_{\bf u}$ is a subshift of $\mathsf{A}^\mathbb{Z}$. We will use the following result from \cite[Proposition 2.5]{ADE} (also see \cite[Proposition 3.1]{GLPS}).

\begin{proposition}\label{prop:Toe} Let ${\bf u}$ be a generating sequence. Suppose ${\bf u}$ has constant length and is proper and primitive. Then $X_{\bf u}$ is a Toeplitz subshift.
\end{proposition}

We will also use the concept of recognizability for $X_{\bf u}$, where ${\bf u}=(u_{n,i})_{n\in\mathbb{N}, i\leq \ell_n}$ is a generating sequence. We follow \cite{GL} in a review of this notion. 
For every $n\in\mathbb{N}$ and $x\in X_{\bf u}$, there is a way to express $x$ as an bi-infinite concatenation of words in $S_n$, which is called a {\em building} of $x$ from $S_n$. In such a building, an occurrence of any $v\in S_n$ in $x$ as a constituent of the building is called an {\em expected occurrence} of the building. We say that ${\bf u}$ is {\em recognizable} if for every $n\in \mathbb{N}$ and $x\in X_{\bf u}$, there is a unique building of $x$ from $S_n$. 
The notion of recognizability has a useful equivalent formulation. For each $n\in\mathbb{N}$ and $0\leq i\leq \ell_n$, define the {\em cylinder set} 
$ [v_{n,i}]$ to be the set of all $x\in X_{\bf u}$ such that  $v_{n,i}$ occurs in  $x$  at position $0$ as an expected occurrence of  some building of $x$  from $S_n$.  Denote the shift map on $X_{\bf u}$ by $\sigma$. Then ${\bf u}$ is recognizable if and only if for every $n\in\mathbb{N}$, the set
$$ \{ \sigma^k[v_{n,i}]\colon 0\leq i\leq \ell_n, 0\leq k<|v_{n,i}|\} $$
is a clopen partition of $X_{\bf u}$. This equivalent formulation is easy to work with when we consider invariant measures on $X_{\bf u}$. Specifically, if ${\bf u}$ has constant length and is recognizable, and $\mu$ is a $\varphi$-invariant probability measure on $X_{\bf u}$, then for every $n\in\mathbb{N}$, we have 
$$ \mu\left(\bigcup_{0\leq i\leq \ell_n}[v_{n,i}]\right)=\displaystyle\sum_{0\leq i\leq \ell_n}\mu([v_{n,i}])=\displaystyle\frac{1}{|v_{n,0}|} $$
by the equivalent formulation.

We are now ready for the following main theorem of this subsection.

\begin{theorem} \label{thm:Toe} The equivalence relation $=^+$ is Borel reducible to the orbit equivalence for regular $\{0,1\}$-Toeplitz subshifts.
\end{theorem}

The rest of this subsection is devoted to a proof of Theorem~\ref{thm:Toe}.

By \cite[Lemma 4.5]{Kaya}, there is an uncountable Borel subset $A$ of $\mathbb{R}$ such that every finite subset of $A\cup\{1\}$ is $\mathbb{Q}$-linearly independent. We fix such a set $A$. Let $=^+_A$ denote the restriction of $=^+$ on $A^\mathbb{N}$. Then $=^+_A$ is Borel bireducible with $=^+$. It suffices to show that $=^+_A$ is Borel reducible to the orbit equivalence for regular $\{0,1\}$-Toeplitz subshifts.

Fix a bijection $\langle \cdot, \cdot\rangle\colon \mathbb{N}\times\mathbb{N}^+\to\mathbb{N}$. Let $\pi_0\colon \mathbb{N}\to \mathbb{N}$ and $\pi_1\colon \mathbb{N}\to \mathbb{N}^+$ be the unique functions such that for any $n\in\mathbb{N}$, 
$$ \langle \pi_0(n), \pi_1(n)\rangle=n. $$
Given $(a_i)_{i\in\mathbb{N}}\in A^\mathbb{N}$, define $(b_i)_{i\in\mathbb{N}}\in \mathbb{R}^\mathbb{N}$ by letting
$$ b_i=\displaystyle\frac{a_{\pi_0(i)}}{\pi_1(i)}.$$  
The point of this adjustment is so that all $\mathbb{Z}$-linear combinations of $(b_i)_{i\in\mathbb{N}}$ are exactly all $\mathbb{Q}$-linear combinations of $(a_i)_{i\in\mathbb{N}}$.

For notational simplicity, we will inductively define a generating sequence 
$$ {\bf v}=(v_{n,j})_{n\geq 1, 0\leq j\leq n} $$
which has constant length and is proper and primitive, 
along with a doubly indexed sequence of real numbers
$$ {\bf c}=(c_{n,j})_{n\geq 1, 0\leq j\leq n} $$
with $(b_i)_{i\in\mathbb{N}}$ as the parameter. Granting the construction, $X_{\bf v}$ is a Toeplitz subshift by Proposition~\ref{prop:Toe}. The composed map
$$ (a_i)_{i\in\mathbb{N}}\mapsto (b_i)_{i\in\mathbb{N}}\mapsto {\bf v}\mapsto X_{\bf v} $$
will witness that $=^+_A$ is Borel reducible to the orbit equivalence for regular Toeplitz subshifts. It will be routine to check that this map is Borel. 

For $n=1$, take $c_{1,1}$ to be an arbitrary element of $(b_0+\mathbb{Q})\cap(1/4,3/4)$, $c_{1,0}=1-c_{1,1}$, $v_{1,0}$ to be the word $0$, and $v_{1,1}$ to be the word $1$. 

For $ n\ge2$, suppose the following inductive hypotheses hold: for all $1\le m< n$, we have defined words $v_{m,j}$ and real numbers $c_{m,j}$ for $0\le j\le m$ with the following properties:
\begin{enumerate}
\item[(1)] The sequence $(v_{m,j})_{1\leq m<n, 0\leq j\leq m}$ satisfies the conditions of constant length and being proper,  primitive and recognizable; 
\item[(2)] For every $1\le m< n$ and $0\leq j\leq m$, $|v_{m,j}|$ is a multiple of $m$; 
\item[(3)] For every $1\le m< n$, the cardinality of the set 
$$\left\{ 0\le i <|v_{m,0}| \colon v_{m,j}(i) \mbox{ are the same letter for all } 0\le j\le m\right\}$$
is at least $(1-1/m)|v_{m,0}|$;

\item[(4)] For every $1\le m< n$ and $0\le j\le m$, $0< c_{m,j}<1$ and $c_{m,m}\in  b_{m-1}+\mathbb{Q}$.
\item[(5)] For $1\le m< m'<n$ and $0\le j\leq m$, $0\leq  j'\le m'$,  let $T^{m,j}_{m',j'}$ be the number of expected occurrences of $v_{m,j}$ in $v_{m',j'}$. Then for $1\leq m<n-1$, $(c_{m+1,j})_{0\le j\le m+1}$ is the unique solution of the system of equations 
\begin{equation*}
\begin{cases}
T^{m,0}_{m+1,0}x_0+T^{m,0}_{m+1,1}x_1+\dots+T^{m,0}_{m+1,m+1}x_{m+1}&=c_{m,0}\\
 \hspace{2em}\vdots\hspace{5em}\vdots \hspace{8em}\vdots&\hspace{2em}\vdots\\
T^{m,m}_{m+1,0}x_0+T^{m,m}_{m+1,1}x_1+\dots+T^{m,m}_{m+1,m+1}x_{m+1}&=c_{m,m}\\
 \hspace{17.5em}  x_{m+1}&=c_{m+1,m+1}
\end{cases}
\end{equation*}
\item[(6)] For $1\le m< m'<n$ and $0\le j\le m$, $0\le  j'\le m'$, we have
$$\left|\frac{T^{m,j}_{m',j'}}{|v_{m',0}|}-c_{m,j}\right|<\frac{1}{m(m+1)|v_{m,0}|}. $$
\end{enumerate}
We need to define words $v_{n,j}$ and real numbers $c_{n,j}$ for $0\le j\le n$. 

Before doing this, we note that inductive hypothesis (5) implies that for any $1\leq m<n-1$ and $0\leq j\leq m$, we have
\begin{equation}\label{eq:5}
\displaystyle\sum_{0\leq i\leq n-1}T^{m,j}_{n-1,i}c_{n-1,i}=c_{m,j}.\tag{*}
\end{equation}
Since $c_{1,0}+c_{1,1}=1=|v_{1,0}|$, it follows from inductive hypothesis (1) and equality (\ref{eq:5}) that
\begin{equation}\label{eq:c}
\displaystyle\sum_{0\leq i\leq n-1}c_{n-1,i}=\displaystyle\frac{1}{|v_{n-1,0}|}. \tag{**}
\end{equation}
We claim that there exists $0<\epsilon_1<1$ such that if 
\begin{equation}\label{eq:a}\left|\frac{T^{n-1,j}_{n,j'}}{|v_{n,0}|}-c_{n-1,j} \right|<\epsilon_1\tag{***}
\end{equation}
for every $0\le j\le n-1$ and $0\le j'\le n$, then inductive hypothesis (6) will continue to hold when $m'=n$.  In fact, suppose $1\leq m<n-1$ and $0\leq j\leq m$. For $0\leq i\leq n-1$ and $0\leq j'\leq n$, let 
$$ \frac{T^{n-1,i}_{n,j'}}{|v_{n,0}|}=c_{n-1,i}+\delta_{i,j'}. $$
Then inequality (\ref{eq:a}) gives $|\delta_{i,j'}|<\epsilon_1$. Denote
$$ M_{m,j}=\displaystyle\sum_{0\leq i\leq n-1}T^{m,j}_{n-1,i}. $$
By inductive hypothesis (5) and in particular equality (\ref{eq:5}), we have
\begin{align*}
\left|\frac{T^{m,j}_{n,j'}}{|v_{n,0}|}-c_{m,j}\right| &=\left|\sum_{0\le i\le n-1}T^{m,j}_{n-1,i}\frac{T^{n-1,i}_{n,j'}}{|v_{n,0}|}-c_{m,j}\right|\\
&=\left|\sum_{0\le i\le n-1}T^{m,j}_{n-1,i}(c_{n-1,i}+\delta_{i,j'})-c_{m,j}\right|\\
&=\left|\sum_{0\le i\le n-1}T^{m,j}_{n-1,i}c_{n-1,i}-c_{m,j}+\sum_{0\leq i\leq n-1}T^{m,j}_{n-1,i}\delta_{i,j'}\right|\\
&= \left| \sum_{0\leq i\leq n-1}T^{m,j}_{n-1,i}\delta_{i,j'}\right| \le \epsilon_1M_{m,j}.
\end{align*}
Thus, if we take 
$$\epsilon_1<\min\left\{\displaystyle\frac{1}{(n-1)n|v_{n-1,0}|}, \displaystyle\frac{1}{m(m+1)|v_{m,0}|M_{m,j}}\colon 1\le m<n-1, 0\le j \le m\right\},$$
then the claim holds. Fix such an $\epsilon_1$.

Take any $\epsilon_2>0$ such that 
$$\epsilon_2<\min\left\{\frac{\epsilon_1}{4}, \frac{c_{n-1,j}}{4}: 0\le j\le n-1\right\}. $$
Let $\epsilon_3>0$. 
Consider the following $(n+1)\times (n+1)$ matrix $U$
$$
\begin{bmatrix}
c_{n-1,0}+\epsilon_2 &c_{n-1,0}-\displaystyle\frac{\epsilon_2 }{n-1} &\cdots&c_{n-1,0}-\displaystyle\frac{\epsilon_2 }{n-1}&c_{n-1,0}-\displaystyle\frac{\epsilon_2 }{n-1}\\ \\
c_{n-1,1}-\displaystyle\frac{\epsilon_2 }{n-1}&c_{n-1,1}+\epsilon_2 &\cdots&c_{n-1,1}-\displaystyle\frac{\epsilon_2 }{n-1}&c_{n-1,1}-\displaystyle\frac{\epsilon_2 }{n-1}\\ \\
\vdots&\vdots&\ddots&\vdots&\vdots\\ \\
c_{n-1,n-1}-\displaystyle\frac{\epsilon_2 }{n-1}&c_{n-1,n-1}-\displaystyle\frac{\epsilon_2 }{n-1}&\cdots&c_{n-1,n-1}+\epsilon_2&c_{n-1,n-1}-\displaystyle\frac{\epsilon_2 }{n-1}\\ \\
0&0&\dots&0&1\\
\end{bmatrix}
$$
and the vector 
$$\vec{c}=[c_{n-1,0},\ c_{n-1,1},\ \dots,\ c_{n-1,n-1},\ \epsilon_3]^T.$$
It is routine to check that $U$ is invertible and therefore the equation $U\vec{x}=\vec{c}$ has a unique solution.
When $\epsilon_3$ converges to $0$, the solution converges to 
$$\left[\frac{1}{n},\cdots,\frac{1}{n},\frac{1}{n},0\right]^T. $$
Therefore we may find an $\epsilon_3 \in b_{n-1} +\mathbb Q$ such that  the solution is in $(0,1)^{n+1}$.

Fixing an $\epsilon_3$ as above, we note that there is $0<\epsilon_4<\epsilon_2$ such that,  if $U'$ is any $(n+1)\times (n+1)$ matrix such that $|U'_{j,i}-U_{j,i}|<\epsilon_4$ for any $0\le j<n$, $0\le i \le n$ and $U'_{n,i}=U_{n,i}$ for any $0\le i\le  n$, then the equation $U'\vec{x}=\vec{c}$ also has a unique solution in $(0,1)^{n+1}$ with the last coordinate exactly $\epsilon_3$. Fix such an $\epsilon_4$.

We conclude that when the  following conditions are satisfied, the system of equations in the inductive hypothesis (5), where $m=n-1$, has a unique solution in $(0,1)^{n+1}$:
\begin{itemize}

\item   for any $0\le j\le n-1$, $0\le i\le n$ and $j\neq i$, 
$$\left|\frac{T^{n-1,j}_{n,i}}{|v_{n,0}|}-c_{n-1,j}+\frac{\epsilon_2}{n-1} \right|<\epsilon_4;$$
\item  for any $0\le j\le n-1$,
$$\left|\frac{T^{n-1,j}_{n,j}}{|v_{n,0}|}-c_{n-1,j}-\epsilon_2 \right|<\epsilon_4;$$
\item $c_{n,n}=\epsilon_3$.
\end{itemize}

For any $0\le j\le n-1$ and $0\le i\le n$, let $\frac{T^{n-1,j}_{n,i}}{|v_{n,0}|}$ be a rational number satisfying the above conditions plus the following conditions:
\begin{itemize}
\item there are positive integers $p_{i,j}$, $q_{i,j}$ such that 
$$\frac{T^{n-1,j}_{n,i}}{|v_{n,0}|}=\frac{p_{i,j}}{n^{q_{i,j}}};$$
\item for any $0\le i\le n$,
$$\sum_{0\le j \le n-1}\frac{T^{n-1,j}_{n,i}}{|v_{n,0}|}=\frac{1}{|v_{n-1,0}|};$$
\item and
$$ \displaystyle\sum_{0\leq j\leq n-1}\min_{0\leq i\leq n}\displaystyle\frac{T^{n-1,j}_{n,i}}{|v_{n,0}|}\geq \displaystyle\frac{1}{n|v_{n-1,0}|}. $$
\end{itemize}
The second condition can be met because of equality (\ref{eq:c}). The last condition can be met because
$$ \displaystyle\frac{T^{n-1,j}_{n,i}}{|v_{n,0}|}\geq c_{n-1,j}-\frac{\epsilon_2}{n-1}-\epsilon_4, $$
$\epsilon_4<\epsilon_2$, and equality (\ref{eq:c}). 
These numbers give rise to an $(n+1)\times (n+1)$ matrix $U'$ such that $|U'_{j,i}-U_{j,i}|<\epsilon_4$ for any $0\le j<n$, $0\le i \le n$ and $U'_{n,i}=U_{n,i}$ for any $0\le i\le  n$. Let $(c_{n,i})_{0\le i\le n}$ be the unique solution of the equation $U'\vec{x}=\vec{c}$. 

Finally, define $(v_{n,i})_{0\leq i\leq n}$ to satisfy the following conditions:
\begin{itemize}
\item for all $0\leq i\leq n$, $|v_{n,i}|=|v_{n,0}|$ and it is a sufficiently large multiple of $n^{q_{i,j}}|v_{n-1,0}|$ so that for any $0\leq j\leq n-1$, $T^{n-1,j}_{n,i}=U'_{j,i}|v_{n,0}|$ is a sufficiently large even number;
\item for each $0\leq i\leq n$ and $0\leq j\leq n-1$, $v_{n-1,j}$ occurs exactly $T^{n-1,j}_{n,i}=U'_{j,i}|v_{n,0}|$ many times in the building of $v_{n,i}$ from $S_{n-1}$;
\item for each $0\leq i\leq n$, the first three terms and the last three terms of the building of $v_{n,i}$ are both $v_{n-1,0}v_{n-1,1}v_{n-1,0}$, and the other occurrences of $v_{n-1,1}$ in $v_{n,i}$ all occur in pairs;
\item letting $L=|v_{n,0}|/|v_{n-1,0}|$, there is a subset $K\subseteq \{1, \dots, L\}$ such that $|K|\geq L/n$ and for any $k\in K$, for all $0\leq i\leq n$, the $k$-th terms of the buildings of $v_{n,i}$ from $S_{n-1}$ are the same.
\end{itemize}
The third condition can be met because $T^{n-1,1}_{n,i}$ are all even by the first condition. The last condition can be met because, letting 
$$ |K|=\left(\displaystyle\sum_{0\leq j\leq n-1}\min_{0\leq i\leq n}\displaystyle\frac{T^{n-1,j}_{n,i}}{|v_{n,0}|}\right)|v_{n,0}|, $$
then $|K|\geq L/n$.

This finishes the definition of $v_{n,i}$ and $c_{n,i}$ for $0\leq i\leq n$. 

It is easily seen that the resulting sequence $(v_{n,i})_{1\leq  i\leq n}$ has constant length and is proper by the third condition above. It is primitive since $T^{n-1,j}_{n,i}>0$ for each $0\leq i\leq n$ and $0\leq j\leq n-1$. The third condition above, together with inductive hypothesis (1), imply that the sequence is recognizable. Thus inductive hypothesis (1) is maintained. It is also clear that inductive hypotheses (2), (4) and (5) are also explicitly maintained, and inductive hypothesis (6) follows from the claim above. For  (3), note that by inductive hypothesis (3), the cardinality of the set
$$ \{0\leq i<|v_{n,0}|\colon \mbox{$v_{n,i}$ are the same letter for all $0\leq i\leq n$}\} $$
is at least 
$$\begin{array}{rl}
 & |K||v_{n-1,0}|+(L-|K|)\left(1-\displaystyle\frac{1}{n-1}\right)|v_{n-1,0}| \\ \\
 \geq & 
 \displaystyle\frac{1}{n}\displaystyle\frac{|v_{n,0}|}{|v_{n-1,0}|}|v_{n-1,0}| + \displaystyle\frac{|v_{n,0}|}{|v_{n-1,0}|}\left(1-\displaystyle\frac{1}{n}\right)\left(1-\displaystyle\frac{1}{n-1}\right)|v_{n-1,0}| \\ \\
= & \left(1-\displaystyle\frac{1}{n}\right)|v_{n,0}|.
  \end{array}
  $$


Now the inductive construction gives a generating sequence
$$ {\bf v}=(v_{n,j})_{n\geq 1, 0\le j\le n}. $$ 
By Proposition~\ref{prop:Toe} and inductive hypotheses (1) and (3), $X_{\bf v}$ is a regular Toeplitz subshift. Let $\mu$ be the unique ergodic measure on $X_{\bf v}$. Then by inductive hypotheses (5) and (6), we have $\mu([v_{n,i}])=c_{n,i}$ 
 for every $n\geq 1$ and $0\le i\le n$. In fact, by inductive hypothesis (5), for any $m>n$, we have 
$$T_{m,0}^{n,i}c_{m,0}+T_{m,1}^{n,i}c_{m,1}+\dots+ T_{m,m}^{n,i}c_{m,m}=c_{n,i}.$$
By inductive hypothesis (6), for any $m'>m$, $0 \le j \le m$, we have 
$$ \left|\frac{T_{m',0}^{m,j}}{|v_{m',0}|}-c_{m,j}\right|< \frac{1}{m(m+1)|v_{m,0}|}.$$ Thus 
\begin{align*}
 \left| \frac{T_{m',0}^{n,i}}{|v_{m',0}|}-c_{n,i} \right| = & \left|\frac{1}{|v_{m',0}|}\sum_{0\le j \le m}T_{m,j}^{n,i}T_{m',0}^{m,j}-c_{n,i}\right| \\
 \leq & \sum_{0\le j \le m} T_{m,j}^{n,i}\left|\frac{T_{m',0}^{m,j}}{|v_{m',0}|}-c_{m,j}\right| \\
<& \frac{1}{m(m+1)|v_{m,0}|}\sum_{0\le j \le m} T_{m,j}^{n,i} < \frac{1}{m}.
\end{align*}
As $m$ and $m'$ go to infinity,  $\frac{T_{m',0}^{n,i}}{|v_{m',0}|}$ goes to $\mu([v_{n,i}])$. Thus $\mu([v_{n,i}])=c_{n,i}$.

By Lemma~\ref{lem:OEkey}, the orbit equivalence for regular Toeplitz subshifts is determined by $\Gamma_{\sigma}$ as a countable set of real numbers. Now it follows from indutive hypothesis (2) that $\mathbb{Q}\subseteq \Gamma_\sigma$. By inductive hypothesis (4), $b_n\in \Gamma_{\sigma}$ for every $n\in\mathbb{N}$. It follows that all $\mathbb{Z}$-linear combinations of $(b_i)_{i\in\mathbb{N}}$ are elements of $\Gamma_{\sigma}$. This further implies that all $\mathbb{Q}$-linear combinations of $(a_i)_{i\in\mathbb{N}}$ are elements of $\Gamma_{\sigma}$. By indutive hypothesis (5), for every $n\geq 1$ and $0\le i\le n$, $c_{n,i}$ is a $\mathbb{Q}$-linear combination of $\{b_0,\cdots, b_n\}$. So $\Gamma_{\sigma}$ is exactly the $\mathbb{Q}$-module generated by $\{1, a_n\colon n\in\mathbb{N}\}$. So the orbit equivalence is determined by this $\mathbb{Q}$-module as a countable set of real numbers. By our assumption for the set $A$, the $\mathbb{Q}$-modules generated by $\{1, a_n\colon n\in\mathbb{N}\}$ and $\{1, a'_n\colon n\in\mathbb{N}\}$ are the same if and only if $(a_n)_{n\in\mathbb{N}}=^+(a'_n)_{n\in\mathbb{N}}$, so our construction is a Borel reduction from $=^+_A$ to the orbit equivalence for regular Toeplitz subshifts.

This completes the proof of Theorem~\ref{thm:Toe}.

\section{Minimal Subshifts of Finite Topological Rank\label{sec:4}}

In this section we consider the orbit equivalence for minimal subshifts of finite topological rank. By the theorem of Downarowicz and Maass \cite{DM}, these are exactly Cantor minimal systems of finite topological rank $n\geq 2$. We first note that this is a subclass of Cantor minimal systems with finitely many ergodic measures.

Recall from \cite{GL} that a Cantor minimal system $(X, \varphi)$ has {\em finite topological rank} if it is topologically conjugate to a Bratteli--Vershik system in which the numbers of vertices at all levels are bounded. Such ordered Bratteli diagrams are also said to be of {\em finite rank}. The least bound in the above definition is the {\em topological rank} of $(X, \varphi)$. To study invariant measures of Cantor minimal systems, it is convenient to use their representation by a nested system of Kakutani--Rohlin partitions given by Herman, Putnam and Skau \cite[Theorem 4.2]{HPS} (also see \cite[Theorem 2.2]{GL}). There is a natural and straightforward correspondence between an ordered Bratteli diagram and the nested system of Kakutani--Rohlin partitions (see \cite[Section 2.5]{GL}). 

More specifically, recall that
 for a minimal Cantor system $(X, \varphi)$, a {\em Kakutani–Rohlin partition} is a
 partition
 $$ \mathcal{P} =\left\{ \varphi^jB(k)\colon 1\leq k\leq d, 0\leq j<h(k)\right\} $$
of $X$ with clopen sets, where $d, h(1), \dots, h(d)$ are positive integers and $B(1), \dots, B(d)$ are
 clopen subsets of $X$ such that
$$\bigcup^d_{k=1} \varphi^{h(k)}B(k) = \bigcup^d_{k=1} B(k). $$
 The set $B(\mathcal{P}) = \bigcup^d_{k=1} B(k)$ is called the {\em base} of $\mathcal{P}$. For $1\leq k\leq d$, the subpartition $T_k =\{\varphi^jB(k)\colon 0\leq  j<h(k)\}$ is the $k$-th {\em tower} of $\mathcal{P}$, which has {\em base} $B(k)$ and {\em height} $h(k)$. If $\mu$ is an invariant measure for $(X, \varphi)$, then for every $1\leq k\leq d$, every set in the $k$-tower $T_k$ has the same $\mu$-measure as the base of $T_k$. Moreover, if $\mu$ is a probability measure, then
$$ \displaystyle\sum_{k=1}^d h(k)\mu(B(k))=1. $$

\begin{lemma}\label{lem:MSN} For any $n\geq 1$, a Cantor minimal system of topological rank $n$ has at most $n$ ergodic measures. In particular, for any $n\geq 2$, a minimal subshift of topological rank $n$ has at most $n$ ergodic measures.
\end{lemma}

\begin{proof}
Let $(X,\varphi)$ be a Cantor minimal system of topological rank $n$. Then there is a nested system of Kakutani--Rohlin partitions $(\mathcal{P}_i)$ satisfying   \cite[Theorem 4.2]{HPS} (also see \cite[Theorem 2.2]{GL}) and so that for every $i\in\mathbb{N}$, $\mathcal{P}_i$ has at most $n$ towers. 

Recall that $\mbox{\rm M}_\varphi$ is the set of all $\varphi$-invariant Borel probability measures of $X$. We define a map $\Phi_i\colon \mbox{\rm M}_\varphi\to [0,1]^{\mathcal{P}_i}$ as follows. For each $\mu\in\mbox{\rm M}_\varphi$ and $i\in\mathbb{N}$, let $\Phi_i(\mu)$ be the restriction of $\mu$ on $\mathcal{P}_i$. Let $S_i$ be  the image of $\Phi_i$. Then $S_i$ is a simplex. Because $\mathcal{P}_i$ has at most $n$ towers, $S_i$ is at most $(n-1)$-dimensional and therefore has at most $n$ extreme points.  

Now consider the inverse system
$$ Y_0\stackrel{\theta_0}{\longleftarrow} Y_1\stackrel{\theta_1}{\longleftarrow} \cdots \stackrel{\theta_{i-1}}{\longleftarrow}Y_i\stackrel{\theta_i}{\longleftarrow} Y_{i+1}\stackrel{\theta_{i+1}}{\longleftarrow}\cdots\cdots $$
where for each $i\in\mathbb{N}$, $Y_i=[0,1]^{\mathcal{P}_i}$ and for each $m\in [0,1]^{\mathcal{P}_{i+1}}$ and $A\in \mathcal{P}_i$,
$$ \theta_i(m)(A)=\displaystyle\sum\left\{ m(B)\colon B\subseteq A, B\in \mathcal{P}_{i+1}\right\}. $$
Then for $\mu\in \mbox{\rm M}_\varphi$, $\mu$ is the inverse limit of $(\Phi_i(\mu))_{i\in \mathbb{N}}$.

We verify that $\mbox{\rm M}_\varphi$ has at most $n$ ergodic measures. Toward a contradiction, assume there are at least $n+1$ ergodic measures $\mu_0, \dots, \mu_{n}$ for $(X, \varphi)$. Then there is a large enough $i\in\mathbb{N}$ such that $\Phi_i(\mu_j)=\theta_{\infty,i}(\mu_j)$, $0\leq j\leq n$, are pairwise distinct and linearly independent. Let $S$ be the simplex generated by $\mu_0, \dots, \mu_n$. Then $\theta_{\infty,i}$ is an injection on $S$. Since $S$ is at least $n$-dimensional, so is $\theta_{\infty,i}(S)=S_i$, a contradiction.
\end{proof}

\subsection{Upper bounds}
In this subsection we give some upper bounds for the orbit equivalence for minimal subshifts of finite topological rank. For $N\geq 2$, denote the orbit equivalence for minimal subshifts of topological rank $N$ by $R_N$. We will show that for $N\geq 2$, $R_N$ is Borel reducible to a countable Borel equivalence relation, and furthermore, for $N=2$, $R_2$ is Borel reducible to the orbit equivalence relation induced by an action of a countable amenable group. In view of the recent result of Naryshkin and Vaccaro \cite{NV}, it follows that $R_2$ is hyperfinite if and only if it is treeable.





For $N\geq 2$ and $1\leq K\leq N$, let $\mathcal{R}_{N,K}$ be the class of all minimal subshifts of topological rank $N$ and with exactly $K$ ergodic measures. By \cite[Theorem 3.3]{GL} and Corollary~\ref{cor:KEM}, $\mathcal{R}_{N,K}$ is a Borel subset of $M(\mathcal{C})$. By Theorem~\ref{thm:GPS} (3), $\mathcal{R}_{N,K}$ is invariant under $R_N$.  We let $R_{N,K}$ denote the restriction of $R_N$ on $\mathcal{R}_{N,K}$. For $L\geq 1$, let $\mathcal{R}_{N,K,L}$ be the subset of $\mathcal{R}_{N,K}$ consisting of all minimal subshifts $(X, \varphi)$ where the maximal cardinality of a $\mathbb{Q}$-linearly independent set of $\Gamma_\varphi$ has size $L$. It is easily seen that $\mathcal{R}_{N,K,L}$ is a Borel subset of $R_{N,K}$. By Theorem~\ref{thm:GPS} (2), $\mathcal{R}_{N,K,L}$ is invariant under $R_N$. Let $R_{N,K,L}$ denote the restriction of $R_N$ on $\mathcal{R}_{N,K,L}$.

\begin{lemma}\label{lem:NKL} $\mathcal{R}_{N,K,L}=\varnothing$ for $L>N$.
\end{lemma}

\begin{proof}
By a theorem of Donoso, Durand, Maass and Petite \cite[Theorem 4.1 (1)]{DDMP}, for any minimal subshift $(X,\varphi)\in R_N$, there exists a primitive and recognizable generating sequence ${\bf v}=(v_{n, i})_{n\in\mathbb{N},0\le i< N}$ so that $(X, \varphi)$ is topologically conjugate to $(X_{\bf v}, \sigma)$. In particular, if $(X, \varphi)\in \mathcal{R}_{N,K,L}$ then so is $(X_{\bf v},\sigma)$ and we have $\Gamma_\varphi=\Gamma_\sigma$. Let $\mu_1, \dots, \mu_K$ be the ergodic measures of $(X_{\bf v}, \sigma)$. Let $\{f_1, f_2,\dots, f_L\}$ be a set of $C(X,\mathbb Z)$. Suppose that the set
$$ \left\{ \left(\int f_\ell\, d\mu_1, \dots, \int f_\ell\, d\mu_K\right)\colon 1\leq \ell\leq L\right\} $$
is $\mathbb{Q}$-linearly independent. Note that there is $n\in\mathbb{N}$ such that each $f_\ell$, $1\leq \ell\leq L$, is a $\mathbb{Z}$-linear combination of elements in the set
$$ \left\{ 1_{\sigma^j([v_{n,i}])}\colon 0\leq i<N, 0\leq j<|v_{n,i}|\right\}, $$
where for any $A\subseteq X_{\bf v}$, $1_A$ denotes the characteristic function of $A$. It follows that for each $1\leq \ell\leq L$, 
$$ \left(\int f_\ell\, d\mu_1, \dots, \int f_\ell\, d\mu_K\right) $$
is a $\mathbb{Z}$-linear combination of elements in the set
$$ \left\{ \big(\mu_1([v_{n,i}]), \dots, \mu_K([v_{n,i}])\big)\colon 0\leq i<N\right\}. $$
This implies that $L\leq N$. 
\end{proof}


Now for any $1\le K, L\le N$, let $E_{N,K,L}$ be the equivalence relation on 
$$ \Sigma=(\mathbb{R}^K)^L\times 2^{\mathbb{Q}^L} $$
defined as follows: for $\vec{a}_1,\cdots, \vec{a}_L, \vec{b}_1,\cdots, \vec{b}_L\in \mathbb{R}^K$ and $p,q\in 2^{\mathbb{Q}^k}$, 
$$ \mbox{$(\vec{a}_1,\dots, \vec{a}_L,p)$ and $(\vec{b}_1,\dots, \vec{b}_L,q)$ are $E_{N,K,L}$-equivalent } $$
if and only if there is $\phi \in \mbox{\rm Sym}(K)$ and $M\in \mbox{\rm GL}(L,\mathbb{Q})$ such that for every $\vec{r}\in\mathbb{Q}^L$,
$$M(\phi(\vec{a}_1),\dots, \phi(\vec{a}_L))^T=(\vec{b}_1,\dots, \vec{b}_L)^T \mbox{ and } q(\vec{r})=p(\vec{r}M),$$ 
where for $\vec{c}=(c_1,\dots, c_K)\in \mathbb{R}^K$, $\phi(\vec{c})=\phi(c_1,\dots, c_K)=(c_{\phi(1)},\dots, c_{\phi(K)})$.

It is easily seen that $E_{N,K,L}$ is a countable Borel equivalence relation induced by a continuous action of $\mbox{\rm GL}(L,\mathbb{Q})\times \mbox{\rm Sym}(K)$ on $\Sigma$. We show that $R_{N,K,L}$ is Borel reducible to $E_{N,K,L}$.

For any $(X,\varphi)\in \mathcal{R}_{N,K,L}$, let $\{ {\vec{a}}^\varphi_1, \dots, \vec{a}^\varphi_L\}$ be  a maximal $\mathbb{Q}$-linearly independent set of $\Gamma_\varphi$. Such a set can be obtained in a Borel way. Then for every $f \in C(X,\mathbb Z)$, there exists $\vec{r}\in\mathbb{Q}^L$ such that 
\begin{equation}\label{eq:m}
\left(\int f\, d\mu_1, \dots, \int f\, d\mu_K\right)=\vec{r}(\vec{a}^\varphi_1, \dots, \vec{a}^\varphi_L)^T.
\tag{$\dagger$}
\end{equation} 
Define  $p^\varphi\in2^{\mathbb{Q}^k}$ by letting $p^\varphi(\vec{r})=1$ if and only if there is  $f \in C(X,\mathbb{Z})$ such that identity (\ref{eq:m}) holds. Thus  we have defined a map $\pi\colon \mathcal{R}_{N,K,L} \to \Sigma$ by
$$ (X,\varphi)\mapsto (\vec{a}^\varphi_1, \dots, \vec{a}^\varphi_L, p^\varphi). $$
It is routine to see that $\pi$ is a Borel map. We verify that $\pi$ witnesses that $R_{N,K,L}$ is Borel reducible  to $E_{N,K,L}$.

First suppose $(X,\varphi), (Y, \psi)\in \mathcal{R}_{N,K,L}$ are orbit equivalent. By Lemma~\ref{lem:OEkey} there is $\phi\in\mbox{\rm Sym}(K)$ such that $\Gamma_\varphi=\phi(\Gamma_\psi)$. Thus $\phi(\vec{a}^\varphi_1), \dots, \phi(\vec{a}^\varphi_L)$ and $\vec{a}^\psi_1, \dots, \vec{a}^\psi_L$ can represent each other as $\mathbb{Q}$-linear combinations. This means that there is $M\in\mbox{\rm GL}(L, \mathbb{Q})$ such that  
$$ M(\phi(\vec{a}^\varphi_1),\dots, \phi(\vec{a}^\varphi_L))^T=(\vec{a}^\psi_1,\dots, \vec{a}^\psi_L)^T. $$
Moreover, since 
$$ \Gamma_\psi=\left\{ \vec{r}(\vec{a}^\psi_1,\dots, \vec{a}^\psi_L)^T\colon p^\psi(\vec{r})=1\right\}, $$
we have that $p^\varphi(\vec{r})=p^\psi(\vec{r}M)$.  This means  that $\pi(X,\varphi)E_{N,K,L} \pi(Y,\psi)$.
Conversely, if $\pi(X,\varphi)E_{N,K,L} \pi(Y,\psi)$, then there is $\phi\in\mbox{\rm Sym}(K)$ such that $\Gamma_\varphi=\phi(\Gamma_\psi)$. By Lemma~\ref{lem:OEkey}, $(X,\varphi)$ and $(X',\varphi)$ are orbit equivalent.

Thus we have proved the following theorem.

\begin{theorem}\label{thm:rn} For any $N\geq 2$, $R_N$ is virtually countable.
\end{theorem}



Our proof gives the following immediate corollary.

\begin{corollary}\label{cor:L1} For any $N\geq 2$ and $1\leq K\leq N$, $R_{N,K,1}$ is virtually hyperfinite.
\end{corollary}

\begin{proof} In this case $E_{N,K,1}$ is given by an action of the group $\mbox{\rm GL}(1,\mathbb{Q})\times \mbox{\rm Sym}(K)$ on $\Sigma$. Note that $\mbox{\rm GL}(1,\mathbb{Q})$ is the multiplicative group of $\mathbb{Q}\setminus\{0\}$, hence is abelian. The action induces an action of $\mbox{\rm GL}(1,\mathbb{Q})$ on $\Sigma$. Let $F$ be the equivalence relation induced by the action of $\mbox{\rm GL}(1, \mathbb{Q})$. Then $F\subseteq E_{N, K, 1}$ is a subequivalence relation with the property that every $E_{N,K,1}$-class contains finitely many $F$-classes.  Since $\mbox{\rm GL}(1,\mathbb{Q})$ is abelian, by a theorem of Gao and Jackson \cite[Corollary 8.2]{GJ}, $F$ is hyperfinite. It follows from a theorem of Jackson, Kechris and Louveau \cite[Proposition 1.3 (vii)]{JKL} that $E_{N,K, 1}$ is hyperfinite.
\end{proof}



By the same technique, and with a bit more care, we can handle the case $N=2$ and get a nontrivial upper bound.

\begin{theorem} For any $N\geq 2$ and $1\leq K\leq N$, $R_{N, K, 2}$ is Borel reducible to an orbit equivalence relation induced by a Borel action of a countable amenable group. In particular, $R_2$ is Borel reducible to an orbit equivalence relation induced by a Borel action of a countable amenable group. Consequently, $R_2$ is virtually amenable. 
\end{theorem}

\begin{proof} The key point is to note that if $(X, \varphi)\in \mathcal{R}_{N,K, L}$, then we may always choose $1^K$ to be an element of a maximally $\mathbb{Q}$-linearly independent subset of $\Gamma_\varphi$. Thus if $1\leq K\leq 2$ and $L=2$, then we may assume $1^K$ and $\vec{a}^\varphi\in \mathbb{R}^K$ form a maximally $\mathbb{Q}$-linearly independent subset of $\Gamma_\varphi$. 

Define an equivalence relation $Q_K$ on $\mathbb{R}^K\times 2^{\mathbb{Q}^2}$ as follows. For $\vec{a}, \vec{b}\in \mathbb{R}^K$ and $p,q\in 2^{\mathbb{Q}^2}$, 
$$ \mbox{ $(\vec{a}, p)$ and $(\vec{b}, q)$ are $Q_K$-equivalent} $$
if and only if there are $\phi\in\mbox{\rm Sym}(K)$ and $r_0, s_0\in \mathbb{Q}$ with $r_0\neq 0$ such that
$$ r_0\phi(\vec{a})+s_01^K=\vec{b} $$
and for any $r, s\in \mathbb{Q}$,
$$ p(r, s)=q(rr_0, rs_0+s). $$

Note that $Q_K$ is the orbit equivalence relation induced by a continuous action of the group
$$ G=(\mbox{\rm GL}(1, \mathbb{Q})\ltimes \mathbb{Q})\times \mbox{\rm Sym}(K), $$
where $\mathbb{Q}$ is the additive group and the semi-direct product is defined by
$$ (r, s)(r',s')=(rr', rs'+s) $$
for $r,r'\in\mathbb{Q}\setminus\{0\}$ and $s, s'\in\mathbb{Q}$. Since $G$ is solvable of rank $2$, it is amenable. 

By a similar argument as in the proof of Theorem~\ref{thm:rn}, we have that $R_{N, K, 2}$ is Borel reducible to $Q_K$. 

The second part of the theorem follows from Lemma~\ref{lem:NKL} and Corollary~\ref{cor:L1}.
\end{proof}

\subsection{Lower bounds}
In this final subsection we give some lower bounds for $R_N$, $N\geq 2$, in the Borel reducibility hierarchy. Along the way we also construct minimal subshifts of topological rank exactly $N$, which seems to be new. The minimal subshifts we construct will again be Toeplitz subshifts.

For $N\ge2$, let $F_N$ be the equivalence relation on $\mathbb{R}^{N-1}$ such that $$(x_1,\cdots,x_{N-1})F_N (y_1,\cdots,y_{N-1})$$ if and only if there is $M\in \mbox{\rm GL}(N,\mathbb{Q})$ such that
$$(x_1,\cdots,x_{N-1},1)M=(y_1,\cdots,y_{N-1},1). $$

Let $\mathsf{A}_N=\{1, \dots, N\}$. We will show that for any $N\ge2$, $F_N$ is Borel reducible to the orbit equivalence for $\mathsf{A}_N$-Toeplitz subshifts of topological rank $N$.

Fix the parameter $(x_1,\cdots,x_{N-1})\in\mathbb{R}^{N-1}$. Take $0<y_i\le1/n$ such that $y_i-x_i\in\mathbb{Q}$ for every $1\le i\leq N-1$. We will define a generating sequence ${\bf v}=(v_{n,i})_{n\in\mathbb{N},1\le i\le N}$ so that the following hold:
\begin{itemize}
\item $(X_{\bf v},\sigma)$ is a regular $\mathsf{A}_N$-Toeplitz subshift;
\item letting $\mu$ be the unique $\varphi$-invariant measure on $X_{\bf v}$, we have $y_i=\mu([i])$ for $1\le i\leq N-1$;
\item and
 $$\Gamma_\sigma=\left\{(q_1,\cdots,q_N)(x_1,\cdots,x_{N-1},1)^T\colon (q_1,\dots, q_N)\in \mathbb{Q}^N\right\}.$$
\end{itemize}
Granting this, we would obtain a map witnessing that $F_N$ is Borel reducible to the orbit equivalence for $\mathsf{A}_N$-Toeplitz subshifts.  In fact, if $$(x_1,\cdots,x_{N-1}), (x'_1,\cdots,x'_{N-1})\in\mathbb{R}^{N-1},$$ then by the last condition of the construction, they are $F_N$-equivalent if and only if
$$\begin{array}{c} \{(q_1,\cdots,q_N)(x_1,\cdots,x_{N-1},1)^T\colon (q_1,\dots, q_N)\in\mathbb{Q}^N\} \\ =\{(q_1,\cdots,q_N)(x'_1,\cdots,x'_{N-1},1)^T\colon (q_1,\dots, q_N)\in\mathbb{Q}^N\},
\end{array}$$ 
which holds if and only if $(X_{\bf v},\sigma)$ and $(X_{{\bf v}'}, \sigma)$ are orbit equivalent by Lemma~\ref{lem:OEkey}.

Next we define $(v_{n,i})_{n\in\mathbb{N},1\le i\le N}$ by induction on $n$. For $n=0$, let $v_{0,i}=i$ for every $1\le i\le N$. Let $c_{0,j}=y_j$ for $1\le j\le N-1$ and 
$$c_{0,N}=1-\sum_{1\le j\le N-1}y_j.$$ Let $h_0=1$. Let $h_1\ge1$ be a large enough even number and take $k_{1,i}\ge1$ for each $1\le i\le N$ to be an even number such that $$c_{0,i}-\frac{1}{2N}< \displaystyle\frac{k_{1,i}}{h_1}< c_{0,i}.$$ 
Let 
$$r_1=h_1-\sum_{1\le i\le N} k_{1,i}=h_1\left(1-\sum_{1\leq i\leq N}\frac{k_{1,i}}{h_1}\right). $$
Then 
$$ 0=h_1\left(1-\sum_{1\leq i\leq N} c_{0,i}\right)<r_1<\frac{h_1}{2}. $$
We can construct $(v_{1,i})_{1\le i\le N}$ such that for every $1\le i\le N$,
\begin{itemize}
\item $v_{1,i}$ is bulit from $S_0=\{v_{0,j}\colon 1\le j\le N\}$;
\item $|v_{1,i}|=h_1$;
\item for each $1\leq j\leq N$, the number of (expected) occurrences of $v_{0,j}=j$ in $v_{1,i}$ is $k_{1,j}+r_1\delta(i,j)$, where 
$$ \delta(i,j)=\left\{\begin{array}{ll} 0, & \mbox{ if $i\neq j$,} \\
1, & \mbox{ if $i=j$;}\end{array}\right.
$$
\item the first and last three terms of the building of $v_{1,i}$ from $S_0$ are both $v_{0,1}v_{0,2}v_{0,1}$, and the other occurrences of $v_{0,2}$ occur in pairs.
\end{itemize}
For every $1\le i\le N$, let 
$$c_{1,i}=\displaystyle\frac{1}{r_{1}}\left(c_{0,i}-\displaystyle\frac{k_{1,i}}{h_{1}}\right).$$ 

In general, suppose we have defined $v_{m,i}$, $h_m$, $k_{m,i}, r_m$ and $c_{m,i}$  for $1\le m\le n, 1\le i\le N$ such that the following hold:
\begin{enumerate}
\item For each $1\leq m\leq n$ and $1\leq i\leq N$, $v_{m,i}$ is built from $S_{m-1}=\{v_{m-1, i}\colon 1\leq i\leq N\}$;
\item The sequence $(v_{m,i})_{0\leq m\leq n, 1\leq i\leq N}$ satisfies the conditions of constant length $(h_m)_{0\leq m\leq n}$ and being proper, primitive and recognizable;
\item For every $1\leq m\leq n$, there is a subset $J$ of $\{0,\dots, h_m-1\}$ with $|J|\geq (1-1/m)h_m$ such that for any $j\in J$, $v_{m,i}(j)$ take the same value for all $1\leq i\leq N$;

\item Letting $T^{m,i}_{m', i'}$ be the number of expected occurrences of $v_{m,i}$ in $v_{m',i'}$ for $0\leq m<m'\leq n$ and $1\le i, i'\leq N$, then for all $0\le m<n$ and $1\leq i,j\leq N$, we have
$$ T^{m,j}_{m+1,i}=k_{m+1,j}+r_{m+1}\delta(i,j) $$ 
and 
$$c_{m+1,i}=\displaystyle\frac{1}{r_{m+1}} \left(c_{m,i}-\displaystyle\frac{k_{m+1,i}}{h_{m+1}}\right); $$ 

\item For each $1\le m\le n$ and $1\leq i\leq N$, we have
$$ \displaystyle\frac{k_{m,i}}{h_{m}}<c_{m-1,i}, \ r_{m}=\displaystyle\frac{h_{m}}{h_{m-1}}-\displaystyle\sum_{1\le i\le N}k_{m,i}, $$ $h_m$ and $r_m$ are positive multiples of $m$, and
$$\sum_{1\le i\le N}c_{m,i}=\displaystyle\frac{1}{h_m}; $$

\item For every $0\le m< m'\le n$ and $1\le i,i'\le  N$, 
$$\left|\displaystyle\frac{T^{m,i}_{m',i'}}{h_{m'}}-c_{m,i}\right|\le\frac{1}{2^{m'}}.$$
\end{enumerate}

Before defining $(v_{n+1,i})_{1\leq i\leq N}$ we note that inductive hypotheses (4) and (5) imply that
for any $0\leq m< n$ and $1\leq j\leq N$, 
$$ c_{m,j}=\displaystyle\sum_{1\leq i\leq N} T^{m,j}_{m+1,i}c_{m+1,i}. $$
From this it follows that for any $0\leq m<m'\leq n$ and $1\leq j\leq N$, 
$$ c_{m,j}=\displaystyle\sum_{1\leq j'\leq N} T^{m,j}_{m',j'}c_{m',j'}. $$


Similar to the claim in the proof of Theorem~\ref{thm:Toe}, we get an $\epsilon>0$ such that the following holds. 
\begin{quote}
Suppose each $v_{n+1,i}$, $1\leq i\leq N$, is built from $S_n=\{v_{n,i}\colon 1\leq i\leq N\}$, with $|v_{n+1, i}|=h_{n+1}$. Suppose for any $1\leq m\leq n$ and $1\leq i,j\leq N$, the number of expected occurrences of $v_{m,j}$ in $v_{n+1,i}$ is $T^{m,j}_{n+1,i}$. If for all $1\leq i,j\leq N$, we have
$$ \left|\displaystyle\frac{T^{n,j}_{n+1,i}}{h_{n+1}}-c_{n,j}\right|<\epsilon, $$
then for all $1\leq m\leq n$ and $1\leq i,j\leq N$,
$$ \left|\displaystyle\frac{T^{m,j}_{n+1,i}}{h_{n+1}}-c_{m,j}\right|<\displaystyle\frac{1}{2^{n+1}}. $$
\end{quote}

Now define $v_{n+1,i}$, $h_{n+1}$, $k_{n+1,i}$, $r_{n+1}$ and $c_{n+1,i}$ for $1\leq i\leq N$ to satisfy the following conditions:
\begin{itemize}
\item[(a)] For each $1\leq i\leq N$, $v_{n+1,i}$ is built from $S_n$ and has length $h_{n+1}$;
\item[(b)] $h_{n+1}$ is a sufficiently large multiple of $(n+1)h_n$;
\item[(c)] For each $1\leq i\leq N$, $k_{n+1,i}$ is a sufficiently large positive multiple of $2(n+1)$ with 
$$ c_{n,i}-\displaystyle\frac{\epsilon}{N}<\displaystyle\frac{k_{n+1,i}}{h_{n+1}}<c_{n,i};$$
\item[(d)] $r_{n+1}$ is a sufficiently large positive multiple of $2(n+1)$ so that
$$ r_{n+1}=\displaystyle\frac{h_{n+1}}{h_n}-\displaystyle\sum_{1\leq i\leq N} k_{n+1, i}; $$
\item[(e)] For each $1\leq i\leq N$,
$$c_{n+1,i}=\displaystyle\frac{1}{r_{n+1}} \left(c_{n,i}-\displaystyle\frac{k_{n+1,i}}{h_{n+1}}\right); $$ 
\item[(f)] For each $1\leq i,j\leq N$, letting
 $$ T^{n,j}_{n+1,i}=k_{n+1,j}+r_{n+1}\delta(i,j), $$ 
then $T^{n,j}_{n+1,i}$ is the number of expected occurrence of $v_{n,j}$ in the building of $v_{n+1,i}$ from $S_n$;
\item[(g)] For each $1\leq i\leq N$, the first and last three terms of the building of $v_{n+1,i}$ from $S_n$ are $v_{n,1}v_{n,2}v_{n,1}$, and the other expected occurrences of $v_{n,2}$ in $v_{n+1,i}$ occur in pairs;
\item[(h)] There is a subset $K\subseteq\{1,\dots, h_{n+1}/h_n\}$ such that $|K|\geq h_{n+1}/[(n+1)h_n]$ and for any $k\in K$, the $k$-th terms of the buildings of $v_{n+1,i}$ from $S_n$ are the same for all $1\leq i\leq N$.
\end{itemize}
Conditions (a)--(f) are straightforward to meet. Condition (g) can be met because for any $1\leq i,j\leq N$, $T^{n,j}_{n+1,i}$ is an even number by conditions (c) and (d). Condition (h) can be met following a similar argument as in the proof of Theorem~\ref{thm:Toe}. Also similarly, the inductive hypotheses are maintained by these definitions. 

This finishes the definition of the generating sequence ${\bf v}=(v_{n,i})_{n\in\mathbb{N}, 1\leq i\leq N}$. 

Now inductive hypotheses (1)--(3) guarantees that $(X_{\bf v}, \sigma)$ is a regular Toeplitz subshift. Let $\mu$ be the unique $\sigma$-invariant measure of $X_{\bf v}$. Then similar as before we have $\mu([v_{n,i}])=c_{n,i}$ for all $n\in\mathbb{N}$ and $1\leq i\leq N$. In particular, we get $\mu([v_{0,j}])=\mu([j])=c_{0,j}=y_j$ for $1\leq j\leq N-1$. By condition (b), we have $\mathbb{Q}\subseteq \Gamma_\sigma$. By condition (e), we have $y_j\mathbb{Q}\subseteq \Gamma_\sigma$ for each $1\leq j\leq N-1$. Therefore,
 $$\Gamma_\sigma=\left\{(q_1,\cdots,q_N)(y_1,\cdots,y_{N-1},1)^T\colon (q_1,\dots, q_N)\in \mathbb{Q}^N\right\}. $$
By the definition of $(y_j)_{1\leq j\leq N-1}$, we have
$$\begin{array}{c} \{(q_1,\cdots,q_N)(x_1,\cdots,x_{N-1},1)^T\colon (q_1,\dots, q_N)\in\mathbb{Q}^N\} \\ =\{(q_1,\cdots,q_N)(y_1,\cdots,y_{N-1},1)^T\colon (q_1,\dots, q_N)\in\mathbb{Q}^N\}.
\end{array}$$ 
Therefore, we finally get
 $$\Gamma_\sigma=\left\{(q_1,\cdots,q_N)(x_1,\cdots,x_{N-1},1)^T\colon (q_1,\dots, q_N)\in \mathbb{Q}^N\right\} $$
 as desired.

We now compute the exact topological rank of $(X_{\bf v},\sigma)$. First, by a result of Donoso, Durand, Maass and Petite \cite[Proposition 4.5]{DDMP} (also see \cite[Proposition 3.6 (ii)]{GLPS}), there is an ordered Bratteli diagram of rank $N$ whose Bratteli--Vershik system in topologically conjugate to $(X_{\bf v}, \sigma)$. This means that $(X_{\bf v},\sigma)$ has topological rank at most $N$. If we choose the parameter $(x_1,\dots, x_{N-1})\in \mathbb{R}^{N-1}$ so that $x_1, \dots, x_{N-1}$ and $1$ are $\mathbb{Q}$-linearly independent, then $(X_{\bf v},\sigma)\in\mathcal{R}_{N,1,N}$. By Lemma~\ref{lem:NKL}, $(X_{\bf v},\sigma)$ has topological rank exactly $N$. 

We have thus proved the following theorem.

\begin{theorem}\label{thm:FN}  For any $N\ge2$, $F_N$ is Borel reducible to the orbit equivalence for regular $\mathsf{A}_N$-Toeplitz subshifts of topological rank $N$. In particular, $F_N\leq_B R_{N,1}\leq_B R_N$.
\end{theorem}

\begin{lemma} The equivalence relation $F_2$ is not smooth. When $n\ge4$, $F_n$ is not hyperfinite. When $n\ge5$, $F_n$ is not treeable.
\end{lemma}

\begin{proof}
Note that the orbit equivalence relation induced by the right action of ${\rm GL}(n-1,\mathbb{Q})$ on $\mathbb{R}^{n-1}$ such that $M\cdot(x_1,\cdots,x_{n-1})=(x_1,\cdots,x_{n-1})M$ is a subrelation of $F_n$. When $n=2$, this action induces non-smooth orbit equivalence relation. By a recent result of Poulin \cite[Corollary 1.3]{Poulin}, when $n\ge4$, this action induces a non-hyperfinite equivalence relation, and when $n\ge5$, it induces a non-treeable equivalence relation. Since a subrelation of a smooth, hyperfinite or treeable countable Borel equivalence relation is still respectively smooth, hyperfinite or treeable, the conclusions of the lemma hold.
\end{proof}

We have the following immediate corollary.

\begin{corollary} For $n\geq 2$, $R_n$ is not smooth.  For $n\geq 4$, $R_n$ is not virtually hyperfinite. For $n\geq 5$, $R_n$ is not virtually treeable.
\end{corollary}

\subsection{Further remarks}

In Lemma~\ref{lem:MSN} we showed that any minimal subshift of topological rank $n\geq 2$ has at most $n$ ergodic measures.  In our proof of Theorem~\ref{thm:FN}, we constructed uniquely ergodic minimal subshifts of topological rank $n$ for any $n\geq 2$. Thus it is possible for a minimal subshift of topological rank $n\geq 2$ to have strictly less than $n$ ergodic measures. 

Our proof of Theorem~\ref{thm:FN} also yields the following.

\begin{corollary} For any $N\geq 2$, there exists a regular $\mathsf{A}_N$-Toeplitz subshift of topological rank exactly $N$. In particular, for any $N\geq 2$, there exists a uniquely ergodic minimal subshift of topological rank exactly $N$.
\end{corollary}

\end{document}